\documentclass[a4paper, 12pt, twoside]{article}
\usepackage[utf8]{inputenc}
\usepackage{amsmath, amsfonts, amsthm, amssymb, amscd, mathtools, mathrsfs} 
\usepackage{graphicx, float, subcaption}  
\usepackage[dvipsnames]{xcolor} 
\usepackage{lipsum} 
\usepackage{enumerate} 
\usepackage{array} 
\usepackage[Lenny]{fncychap} 
\usepackage{bbold} 
\usepackage{indentfirst} 
\usepackage{wasysym}
\usepackage{bm} 
\usepackage{multirow, multicol, tabularx}
 \usepackage{pdfpages}
 \usepackage{ulem}
 \usepackage{upgreek}
\usepackage[]{mdframed}
\usepackage{mathdots}

\usepackage{geometry}
\newgeometry{top=0.9in, bottom=0.9in, outer=0.8in,inner=0.8in}

\usepackage[skip=10pt, indent=15pt]{parskip}
\usepackage{hyperref}
\hypersetup{
	colorlinks=true,
	linkcolor=black,
	urlcolor=green,
	citecolor=blue,
	linktoc=all
}
\definecolor{UnivBlue}{RGB}{47, 42, 133}
\usepackage{fancyhdr}
\fancypagestyle{mypagestyle}{%
  \fancyhf{}
  \fancyhead[OL]{Marc Peign\'e and Tran Duy Vo}
  \fancyhead[ER]{Functional limit theorem}
  \fancyfoot[C]{\thepage}%
}
 \newtheorem{theo}{\textbf{Theorem}\ }
[section]
\newtheorem{lem}[theo]{\textbf{Lemma}\ }

\newtheorem{prop}[theo]{\textbf{Proposition}\ }

\newtheorem{property}[theo]{\textbf{Property}\ }

\begin{document}
\title{A FUNCTIONAL LIMIT THEOREM FOR LATTICE OSCILLATING RANDOM WALK }
\maketitle
\begin{center}
{\bf Marc Peign\'e, Tran Duy Vo}
$^($\footnote{ 
 Institut Denis Poisson UMR 7013,  Universit\'e de Tours, Universit\'e d'Orl\'eans, CNRS, France. \\ marc.peigne@univ-tours.fr, Tran-Duy.Vo@lmpt.univ-tours.fr}$^)$  	 
     \end{center}
     \vspace{1cm}
     
\centerline{\bf  \small Abstract } 
 The paper is devoted to an invariance principle for Kemperman's model of oscillating random walk on $\mathbb{Z}$. This result appears as an extension of the invariance principal theorem  for classical random walks on $\mathbb Z$ or reflected random walks on $\mathbb N_0$. Relying on some natural  Markov sub-process which takes into account the oscillation of the random walks between $\mathbb Z^-$ and $\mathbb Z^+$, we first construct an aperiodic sequence of renewal operators acting on a suitable Banach space  and then apply a powerful theorem proved by S. Gou\"ezel.
   
  \vspace{0.5cm} 
  
  \noindent Keywords and phrases:   oscillating random walk,  invariance principle,  skew Brownian motion, renewal sequences of operators, Markov chain 


\section{Model and setting}
\subsection{Introduction} 
Consider two independent sequences of i.i.d. discrete random variables $(\xi_n)_{n \geq 1}$ and $(\xi'_n)_{n \geq 1}$, defined on a probability space $(\Omega, \mathcal{F}, \mathbb{P})$ and with respective distributions $\mu$ and $\mu'$. 

 For any fixed $\alpha \in [0,1]$, the oscillating random walk  $\mathcal{X}^{(\alpha)}=(X^{(\alpha)}_n)_{n\geq 0}$  is defined recursively by : $X_0^{(\alpha)}=x$, where $x \in \mathbb Z$ is fixed, and, for $n \geq 0$, 
 \begin{align} \label{system}
X^{(\alpha)}_{n+1}=\left\{
\begin{array}{lll}
X^{(\alpha)}_n+ \xi_n &\text{\rm if } X^{(\alpha)}_n \leq -1, \\  
\eta_n &\text{\rm if } X^{(\alpha)}_n=0,\\
X^{(\alpha)}_n +\xi'_n &\text{\rm if } X^{(\alpha)}_n \geq 1, 
\end{array}
\right.    
 \end{align} 
where $\eta_n:= B_n \xi_n +(1-B_n)\xi'_n$ and $(B_n)_{n\geq 0}$ is a sequence of i.i.d.  Bernoulli random variables (independent of $(\xi_n)$ and $(\xi'_n))$  with $\mathbb{P}[B_i=1]=\alpha=1-\mathbb{P}[B_i=0]$.

When we want to emphasize the dependence in $\mu$ and $\mu'$   of this oscillating process, we denote it by $\mathcal{X}^{(\alpha)}(\mu, \mu')$. 

This  spatially non-homogeneous random walk   was first introduced by \textsc{Kemperman} \cite{Kemperman} to model discrete-time diffusion in one dimensional space with three different media $\mathbb{Z^+}$ and $\mathbb{Z^-}$ and a barrier $\{0\}$. Whenever the process $\mathcal{X}^{(\alpha)}(\mu,\mu')$ stays on the negative half line, its excursion is directed by the jumps $\xi_n$ until it reaches the positive half line; then,  it continues being directed by the jumps $\xi'_n$ until returning in the negative half line and so on. After each visit of the origin, the increment is governed by the  distribution  of $\eta_n$, which is a convex combination of $\mu$ and $\mu'$. Our considering system is referred to as a special case when the barrier is degenerated as a single point; in general context, it may be any determined interval $[a, b] \cap \mathbb{Z}$ which passes through the origin, see \cite{KimLotov} for instance. Another interesting variant has been studied by \textsc{Madras} and \textsc{Tanny} in \cite{MT} dealing with an oscillating random walk with a moving  barrier at some constant speed. Although this basically leads to differences in its long-term behaviour compared to (\ref{system}), we may be able to trace this model back to (\ref{system}) by using some appropriate translations for its random increments. 

In the present paper, we prove an invariant principle for $\mathcal{X}^{(\alpha)}(\mu,\mu')$ towards the skew Brownian motion $(B^\gamma_t)_{t>0}$ on $\mathbb{R}$ with parameter $\gamma \in [0, 1]$. The diffusion $(B^\gamma_t)_{t>0}$ is obtained from the standard Brownian process by
independently altering the signs of the excursions away from  $0$, each
excursion being positive with probability $\gamma$ and negative with probability $1-\gamma$. By \cite{RY}, its heat kernel is given by: for any $x, y \in \mathbb{R}$ and $t>0$,
$$
p^\gamma_t(x,y):=p_t(x,y) + (2\gamma -1) \text{ sign}(y) \,p_t(0, |x|+|y|),
$$
where $p_t(x,y)= \frac{1}{\sqrt{2\pi t}} e^{{-(x-y)^2}/{2t}}$ is the transition density of the Brownian motion.

Throughout this paper, we suppose that the following general assumptions always hold:

\noindent $\bf H1 $ {\it $(\xi_n)_{n \geq 1}$ and $(\xi'_n)_{n \geq 1}$ are independent sequences of i.i.d. $\mathbb{Z}$-valued random variables, with finite variances $\sigma^2$ and $\sigma'^2$, respectively}.

\noindent $\bf H2 $ {\it {Both  distributions $\mu$ and $\mu'$ are  centered (i.e. $\mathbb E[\xi_n]=\mathbb E[\xi'_n]=0$}).}

\noindent $\bf H3 $ {\it Both distributions $\mu$ and $\mu'$  are strongly aperiodic on $\mathbb Z$, i.e. their supports are not included in $b+ a\mathbb{Z} $ for any $a>1$ and $ b \in \{0, \ldots, a-1\}$. }

\noindent $\bf H4$ {\it There exists $\delta>0$ such that $\mathbb{E}[(\xi^+_n)^{3+\delta}] +\mathbb{E}[(\xi'^-_n)^{3+\delta}] < +\infty$, where $\xi^+_n:= \max\{0, \xi_n\}$ and $\xi'^-_n:= \max\{0, -\xi'_n\}$}.

 Let us emphasize that, under  hypotheses {\bf H1}, {\bf H2} and {\bf H3}, the oscillating random walk $\mathcal{X}^{(\alpha)}$ is irreducible and null recurrent on $\mathbb Z$; this property is not stated in \cite{Vo} and we will  detail the argument later (see Property \ref{zaiugegfh}).

We denote by $S=(S_n)_{n \geq 1}$  (resp.  $S'=(S'_n)_{n \geq 1}$) the random walk defined by  $S_0=0$ and 
 $S_n= \xi_1 + \ldots+  \xi_n $ for $n \geq 1$  (resp.  $S'_0=0$ and 
 $S'_n= \xi'_1 + \ldots+ \xi'_n $ for $n \geq 1$). Let  $(\ell_i)_{i \geq 0}$ be the  sequence of strictly ascending ladder epochs associated with $S$ and defined recursively  by $\ell_0=0$ and, for $i \geq 1$, 
$$\ell_{i+1}:= \inf\{k > \ell_i \mid S_k > S_{\ell_i}\}$$ 
(with the convention $\inf \emptyset = +\infty$). 
We also  consider  the   sequence of descending ladder epochs $(\ell'_i)_{i \geq 0}$ of $S'$, defined as follows,
$$\ell'_0=0, \quad \text{and } \quad \ell'_{i+1}:= \inf\{k > \ell'_i \mid S'_k < S'_{\ell'_i}\}, \quad \text{for any } i \geq 1.$$
Under hypothesis $\bf H2$,  it holds 
$\displaystyle \mathbb{P}[\limsup_{n \to +\infty}S_n=+\infty]= \mathbb{P}[\liminf_{n \to +\infty} S_n'=-\infty] =1$;  hence, all the random variables $\ell_i$ and $\ell'_i$ are  $\mathbb{P}$-a.s. finite. In addition, both sequences $(\ell_{i+1}-\ell_i)_{i \geq 0}$ and $(S_{\ell_{i+1}}-S_{\ell_i})_{i \geq 0}$ contain i.i.d. random elements with distributions of $\ell_1$ and $S_{\ell_1}$, respectively; the same property holds for $(\ell'_{i+1}-\ell'_i)_{i \geq 0}$ and $(S'_{\ell'_{i+1}}-S'_{\ell'_i})_{i \geq 0}$. Consequently, processes  $(\ell_i)_{i \geq 0}, (S_{\ell_i})_{i \geq 0},  (\ell'_i)_{i \geq 0}$ and $(S'_{\ell'_i})_{i \geq 0}$ are all  random walks with i.i.d. increments.

We denote $\mu_+$  the distribution of $S_{\ell_1}$ and $ \mathcal{U}_+$   its potential defined by $ \mathcal{U}_+:=  \sum_{n \geq 0} (\mu_+)^{*n}$. Similarly   $\mu'_-$ denotes the distribution of $S'_{\ell'_1}$ and $ \mathcal{U'}_-:= \displaystyle \sum_{n \geq 0} (\mu'_-)^{*n}$. 

In particular, the oscillating random walk $\mathcal{X}^{(\alpha)}$ visits $\mathbb Z^-$ and $\mathbb Z^+$ infinitely often ; in order to control the excursions inside each of these these two half lines,  it is natural to consider  the following stopping  times $\tau^S(-x),\ \tau^{S'}(x)$ with $x \geq 1$, associated with $S$ and $S'$ respectively and defined by 
$$
\tau^S(-x):= \inf\{n \geq 1 \mid -x+ S_n \geq 0 \}, \quad {\rm and} \quad \tau^{S'}(x):= \inf\{n \geq 1 \mid x+ S'_n \leq 0 \}.
$$
In the sequel, we  focus on the  ``ascending renewal function''  $h_a$ of $S$ and the ``descending renewal function''  $h'_d$  of $S'$ defined by 
\begin{align*} 
   h_a(x) := \left\{
   \begin{array}{ll} \mathcal{U}_+[0, x]=
  \sum_{i\geq 0} \mathbb{P}[S_{\ell_i}\leq x] & \text{ if } x\geq 0, \\
  0 & \text{ otherwise}, 
   \end{array}
    \right.
        \end{align*}
and  
\begin{align*} 
   h'_d(x) := \left\{
   \begin{array}{ll}
 \displaystyle\mathcal{U}'_-[-x, 0]= \sum_{i\geq 0} \mathbb{P}[S'_{\ell'_i}\geq -x]  & \text{ if }  x\geq 0,\\
 0  & \text{ otherwise}. 
   \end{array}
    \right.
\end{align*}
We denote by $\check{h}_a$ the function $x \mapsto h_a(-x)$, it appears in the definition of the parameter $\gamma$ below.

  Both functions $h_a$ and $h'_d$ are increasing and satisfy  $h_a(x)=O(x)$ and $h'_d(x)=O(x)$. They appear crucially in the quantitative estimates of the fluctuations of $S$ and $S'$ ; see   subsection \ref{section-fluctuations} for precise statements.  
  
  Let us end this paragraph devoted to the presentation of quantities that play an important role in the rest of the paper.
  
  $\bullet$ 
By classical results on 1-dimensional random walks  \cite{Feller},  under hypotheses  {\bf H1}  and {\bf H2},  both constants $\displaystyle 
c=\dfrac{\mathbb{E}[S_{\ell_1}]}{\sigma\sqrt{2\pi}}$  and $\displaystyle c'=\dfrac{\mathbb{E}[-S_{\ell'_1}]}{\sigma'\sqrt{2\pi}} 
$ 
 are finite. 
 
 $\bullet$ Under hypotheses  {\bf H1}, {\bf H2} and {\bf H3}, the ``crossing sub-process'' $\mathcal X^{(\alpha)}_{\bf C}$ which corresponds to the sign changes of the   process $\mathcal X^{(\alpha)}$ is well defined  (see section \ref{crossingprocess}) and  it is positive recurrent  on its unique irreducible class.  We denote by $\nu$ its unique invariant probability measure on $\mathbb Z$.
 
 \subsection{Main result}

  {\bf From now on, we fix $\alpha \in [0, 1]$}  and   consider the  continuous and linearly interpolated version $(X_{nt}) $ of  $\mathcal{X}^{(\alpha)}$, defined by: for any  $n\geq 1$ and   $t\in (0,1]$, 
 $$X^{(\alpha)}_{nt}=  \sum_{i=1}^n \left(X^{(\alpha)}_{[nt]}+(nt-[nt]) \times J_{[nt]+1}\} \right)  \mathbb{1}_{ [\frac{i-1}{n}, \frac{i}{n} [}(t), $$
 where 
 $$
 J_{[nt]+1}:= \left\{
 \begin{array}{lll}
 \xi_{[nt]+1}& if&  X^{(\alpha)}_{[nt]} \leq -1\\
  \eta_{[nt]+1} &if& X^{(\alpha)}_{[nt]} =0\\
 \xi'_{[nt]+1}& if &X^{(\alpha)}_{[nt]} \geq 1.
  \end{array}
  \right.
 $$
We also set  \begin{align*}  
   X^{ (\alpha, n)}(t) := \left\{
   \begin{array}{lll}
\frac{X^{(\alpha)}_{nt}}{\sigma \sqrt{n}} &  \text{if } X _{nt}\leq 0,\\ \\
\frac{X^{(\alpha)}_{nt}}{\sigma' \sqrt{n}} &  \text{if } X _{nt}\geq 0.
   \end{array}
    \right.
\end{align*}
The main result of this paper is the following one.
  \begin{theo} \label{main_paper02} 
  Assume that hypotheses {\bf H1}--{\bf H4} are satisfied. Then, as $n \to +\infty$, the normalized stochastic process $\{X^{(\alpha, n)}(t), t\in [0,1]\}_{n\geq 1}$ converges weakly in the space of continuous function $C([0,1])$ to the skew Brownian motion $W_{\gamma}:=\left\{W_\gamma(t), t\in [0,1] \right\}$ with  parameter $\displaystyle \gamma=\frac{c'\nu(h'_d)}{c\nu(\check{h}_a)+ c'\nu(h'_d)}.$
\end{theo}

 Let us  clarify  the value of  the parameter $\gamma$ in two peculiar cases of (\ref{system}).

$\bullet$   When $\mu=\mu'$, the chain $\mathcal{X}^{(\alpha)}$ is an ordinary random walk on $\mathbb{Z}$ directed by the unique type of jumps $(\xi_n)_{n \geq 1}$ and the limit diffusion $W_\gamma$ is the Brownian motion. In this case, the parameter  $\gamma$ equals $ \frac{1}{2}$ since the sequences of ladder heights $(S_{\ell_i})_{i \geq 1}$ and $(-S'_{\ell'_i})_{i \geq 1}$ coincide.

$\bullet$  When $\mu(x)= \mu'(-x)$ for any $x \in \mathbb Z$, the random walk $\mathcal X^{(\alpha)}$ is the so-called ``anti-symmetric  random walk" (or ``reflected random walk'' as usual), which appears in several works, see for instance \cite{EP} and \cite{PW}. By setting $\xi_n= -\xi'_n$, the behaviour of the chain $\mathcal{X^{(\alpha)}}$ on positive and negative half lines, respectively, are mirror images of each other. Hence, we may "glue" them together to get an unifying Markov chain on $\mathbb{Z}^+ \cup \{0\}$ receiving $\{0\}$ as its reflecting boundary. Accordingly, $\gamma=1$ in this case and it matches perfectly with the  result in \cite{NP}, which states that   the normalized reflected random walk (constructed as above) converges weakly in $C([0,1])$ towards the absolute value of the standard  Brownian motion. 
 
$\bullet$ Notice that the limit process $W_\gamma$  does not depend on $\alpha \in [0, 1]$. Henceforth, we fix $\alpha$ and  set  $\mathcal X^{(\alpha)}= \mathcal X$ in order to simplify the notations.
\subsection{Notations} 
We set $\mathbb{Z}:= \mathbb{Z}^+ \cup \mathbb{Z}^- \cup \{0\},\ \mathbb{N}:=\{1, 2, 3, \ldots\}$ and   $\overline{\mathbb{D}}$ the closed unit ball in $\mathbb C$. Given two positive real sequences ${\bf a}=(a_n)_{n \in \mathbb{N}}$ and ${\bf b}=(b_n)_{n \in  \mathbb{N}}$, we write as usual 
\begin{itemize}
\item $a_n \thicksim b_n$ if $\displaystyle \lim_{n\to \infty} a_n / b_n=1$, 
\item $a_n \approx b_n$ if $\displaystyle\lim_{n\to \infty}(a_n-b_n)=0$,
\item $a_n= O(b_n)$ if $\displaystyle \limsup_{n \to \infty} a_n/b_n < +\infty,$ 
\item $a_n= o(b_n)$ if $\displaystyle\lim_{n \to \infty} a_n/b_n = 0$,
 \item $\bf a \preceq b$ if $a_n \leq c \  b_n$ for some constant $c>0$, 
 \item $\bf a \asymp b$ if $\dfrac{1}{c} \ b_n\leq a_n \leq c \  b_n$ for some constant $c \geq 1.$
\end{itemize}

\vspace{2mm}

{ \it The paper is organized as follows.   In Section  2, we recall some important estimates in the theory of fluctuations of random walks; we introduce  in particular the renewal functions associated with 1-dimensional random walks and relative conditional limit theorems. These helpful tools  appear in Section 4  to compute the  multi-dimensional distribution of the limit process. The center of gravity of the paper is Section 3 where we adapt the approach used in  \cite{NP} in the case of the reflected random walk (with proper adjustments to derive Corollary \ref{convergeH} and to determine the parameter $\gamma$ later on). The last two sections are devoted to the proof of Theorem \ref{main_paper02} and the appendix}. 
\section{Auxiliary results for random walks}
In this section, we present some classical results on fluctuations of random walk on $\mathbb Z$.
\subsection{Asymptotic estimates for  fluctuations of  a random walk} \label{section-fluctuations}

The following statement summarizes   classical results  on fluctuations of random walks which are used below at various places (for instance, see Proposition $11$ in \cite{Doney12}, Theorem A in \cite{Kozlov} et al). Recall that $\displaystyle 
c=\dfrac{\mathbb{E}[S_{\ell_1}]}{\sigma\sqrt{2\pi}}$  and $\displaystyle c'=\dfrac{\mathbb{E}[-S_{\ell'_1}]}{\sigma'\sqrt{2\pi}} 
$.

\begin{lem}\label{equal}   (Asymptotic property)   Under assumptions {\bf H1}-- {\bf H3}, for any $x, y \geq 1$, it holds, as $n\to \infty$, 
\begin{enumerate}[a)]
\item $\mathbb{P}[\tau^{S}(-x)>n] \thicksim 2c \,\frac{h_a(x)}{\sqrt{n}}, \quad  and \quad  \mathbb{P}[\tau^{S'}(x)>n] \thicksim 2c'\, \frac{h'_d(x)}{\sqrt{n}};$
\item $\mathbb{P}[\tau^{S}(-x)>n, -x+S_n=-y] \thicksim\frac{1}{\sigma\sqrt{2\pi}}\, \frac{h_a(x) \,h_d(y)}{n^{3/2}}$,

 \hspace{3cm} and \quad $\mathbb{P}[\tau^{S'}(x)>n, x+S'_n=y] \thicksim\frac{1}{\sigma'\sqrt{2\pi}}\, \frac{h'_d(x)\, h'_a(y)}{n^{3/2}};$
 
 where $h_d$ (resp. $h'_a$) is the descending (resp. ascending) renewal function associated with the random walk $S$ (resp. $S'$).
    \item $\mathbb{P}[\tau^{S}(-x)=n] \thicksim c\, \frac{h_a(x)}{n^{3/2}}, \quad  and \quad \mathbb{P}[\tau^{S'}(x)=n] \thicksim c'\, \frac{h'_d(x)}{n^{3/2}};$
\end{enumerate}
 
\end{lem}

\begin{lem} \label{upperbound}
    {\it  (Upper bound)}  For any $n \geq 1$, it holds
    \begin{enumerate}[a)]
    \item $\mathbb{P}[\tau^{S}(-x)>n] \preceq\, \frac{1+x}{\sqrt{n}},  \quad  and \quad \mathbb{P}[\tau^{S'}(x)>n] \preceq\, \frac{1+x}{\sqrt{n}};$
\item $\mathbb{P}[\tau^{S}(-x)>n, -x+S_n=-y] \preceq\,  \frac{(1+x)(1+y)}{n^{3/2}}$,  

 \hspace{3cm} and \quad $\mathbb{P}[\tau^{S'}(x)>n, x+S'_n=y] \preceq\, \frac{(1+x)(1+y)}{n^{3/2}};$
 
\item $\mathbb{P}[\tau^{S}(-x)=n] \preceq  \frac{1+x}{n^{3/2}}, \quad  and \quad \mathbb{P}[\tau^{S'}(x)=n] \preceq  \frac{1+x}{n^{3/2}}$. 
    \end{enumerate}
\end{lem}

\noindent  
As a direct consequence of $b)$ in Lemmas \ref{equal} and \ref{upperbound}, for any $x\geq 1$ and  $w \geq 0$, 
\begin{align}\label{eventS}
\mathbb{P}[\tau^{S}(-x)=n, -x+S_n=w] \preceq\,  \frac{1+x}{n^{3/2}} \sum_{z \geq w+1} z \mu(z).
\end{align}
Indeed, for any $ n \geq 1$,
\begin{align*} 
\mathbb{P}[\tau^{S}(-x)=n, &-x+S_n =w]  \\
 &=  \sum_{y \geq 1}\mathbb{P}[\tau^{S}(-x)=n, -x+S_{n-1}=-y, -y+\xi_n=w]  
\\
&= \sum_{y \geq 1}\mathbb{P}[\tau^{S}(-x)>n-1, -x+S_{n-1}=-y, -y+\xi_n=w]  
\\
&=  \sum_{y \geq 1}\mathbb{P}[\tau^{S}(-x)>n-1, -x+S_{n-1}=-y]\mu(y+w)
\\
&\preceq  \dfrac{1+x}{n^{3/2}}  \ \underbrace{\sum_{y \geq 1} (1+y)\mu(y+w)}_{\preceq \ \displaystyle \sum_{z\geq w+1} z\mu(z)<+\infty}.
\end{align*}
 Notice also that, more precisely it holds
 $$
 \mathbb{P}[\tau^{S}(-x)=n, -x+S_n =w] \thicksim {h_a(x)\over \sigma \sqrt{2\pi} n^{3/2}}\sum_{y \geq 1} h_d(y) \mu(y+w).
 $$
%
\subsection{Conditional limit theorems}
It is worth remarking some necessary limit theorems which are very helpful for us to control the fluctuation of excursions between two consecutive successive crossing times and contribute significantly to reduce the complexity when dealing with multidimensional distribution of these excursions.
Now, assume that $\mathbb{E}[\xi'_1]=0$ and $\mathbb{E}[(\xi'_1)^2]<+\infty$ and let  $(S'(t))_{t \geq 0}$ be the continuous time process constructed from the sequence $(S'_n)_{n \geq 0}$ by using the linear interpolation between the values at integer points. 

By Lemma $2.3$ in \cite{AKV},  for $x \geq 1$, the rescaled process $\left(\dfrac{x+S'_{[nt]}}{\sigma' \sqrt{n}}, t \in [0, 1]\right)$ conditioning on the event $[ \tau^{S'}(x)>n]$ converges weakly on $C([0, 1], \mathbb{R})$ towards the Brownian meander.
In other words,  for any bounded Lipschitz continuous function $\psi: \mathbb{R} \to \mathbb{R}$ and any $t \in (0,1]$ and $x \geq 1$,
\begin{align}\label{one}
   \lim_{n\to +\infty} \mathbb{E}\left[\psi\left(\dfrac{x+S'_{[nt]}}{\sigma'\sqrt{n} }\right)\mid \tau^{S'}(x)>[nt]\right]=\dfrac{1}{t}  \int_0^{+\infty} \psi(u) u \exp\left({-\frac{u^2}{2t}}\right)du.
\end{align}

 Let us also state the Caravena-Chaumont's result about random bridges conditioned to stay positive in the discrete case. Roughly speaking,  as $n \to +\infty$, for any starting point $x \geq 1$ and any ending point $y \geq 1$, the random bridge of the random walk $S$, starting at $x$, ending  at $y$ at time $n$ and conditioned to stay positive until  time  $n$,  after a linear interpolation and a diffusive rescaling, converges in distribution on $C([0, 1], \mathbb{R})$ towards the normalized Brownian excursion $\mathcal{E}^+$: 
\[
\left(\left(\dfrac{S'_{[nt]}}{\sigma'\sqrt{n}}\right)_{t \in [0, 1]}\mid \tau^{S'}(x)>[nt], S'_n=y\right) \overset{\mathscr{L}}{\longrightarrow}\mathcal{E}^+, \quad \text{as } n\to  +\infty.
\]
More precisely, for any $x, y\geq 1, 0<s<t\leq 1$ and any bounded Lipschitz continuous function $\psi: \mathbb{R} \to \mathbb{R}$,
\begin{align} \label{two}
 \lim_{n\to +\infty} \mathbb{E}&\left[\psi\left(\dfrac{x+S'_{[ns]}}{\sigma'\sqrt{n} }\right)\mid \tau^{S'}(x)>[nt], x+S'_{[nt]}=y\right] \notag
 \\
 &  \hspace{3cm} =
   \int_0^{+\infty} 2\psi(u\sqrt{t})  \exp\left(-\dfrac{u^2}{2\frac{s}{t} \frac{t-s}{t}}\right) \frac{u^2}{\sqrt{2\pi \frac{s^3}{t^3} \frac{(t-s)^3}{t^3}}}du.
\end{align}

\section{Crossing times and renewal theory}
In order to analyse the asymptotic behavior of the process $\mathcal X$, we decompose $X_n$  as a sum of  successive excursions  in   $\mathbb Z^-$ or $\mathbb Z^+$. It is therefore interesting to introduce  the sequence ${\bf C}=(C_k)_{k \geq 0} $ of  ``crossing times'', i.e. times at which the process $\mathcal{X}$ changes its sign: more precisely,  $C_0=0$ and, for any $k \geq 0$, 
\begin{align}\label{crossing_time}
    C_{k+1} :=\left\{ 
    \begin{array}{lll}
 \vspace{2mm}
    \inf \{n > C_k \mid X_{C_k} + (\xi_{C_k + 1} +  \dotsi +\xi_{n}) \geq 0\} &\text{\rm if }  X_{C_k} \leq -1, \\
    \vspace{2mm}
C_k + 1 &\text{\rm if }  X_{C_k} =0,\\
 \vspace{2mm}
 \inf \{n > C_k \mid  X_{C_k} + (\xi'_{C_k + 1} + \dotsi + \xi'_{n}) \leq -1\} &\text{\rm if } X_{C_k} \geq 1. 
 \end{array}
 \right.
\end{align}
Under hypothesis {\bf H2},  the random times $C_k$ are $\mathbb{P}$-a.s. finite and   form a sequence of finite stopping times with respect to the canonical filtration $  (\sigma \left(\xi_{k}, \xi'_{k})\mid k \leq n \right)_{n \geq 1}$.

\subsection{On the crossing sub-process \texorpdfstring{$\mathcal X_{\bf C}$}{X}}\label{crossingprocess}
We denote $\mathcal X_{\bf C}:=(X_{C_k})_{k \geq 0}$ the {\bf crossing sub-process of $\mathcal X$}, which  plays an important role in this paper. 


 \begin{lem}  The sub-process $\mathcal X_{\bf C}$ is a  time-homogeneous Markov chain on $\mathbb{Z}$ with transition kernel 
$\mathcal C = (\mathcal C(x, y))_{x, y \in \mathbb Z}$ given by   
 \begin{equation} \label{kernel_paper02}
  \mathcal{C}(x, y) = \left\{
   \begin{array}{lll}
  \sum_{t = 0}^{-x-1} \mu_+(y-x-t)\, \mathcal{U}_+(t) &  \text{if } x \leq -1 \text{ and } y \geq 0,\\
 \alpha \mu(y) + (1-\alpha) \mu'(y) & \text{if } x=0 \text{ and } y\in \mathbb{Z},\\
  \sum_{t = -x+1}^{0} \mu'_-(y-x-t)\,\ \mathcal{U}'_-(t) & \text{if } x \geq 1 \text{ and } y \leq 0.
   \end{array}
    \right.
\end{equation}
 \end{lem}

 \begin{proof} 
 The Markov property is obvious from the above definition.\\
Now, we compute $\mathcal{C}(x, y)$ for any $x \leq  -1$ and $y \geq 0$ (other cases are similar) as follows. Noticing that the first crossing time $C_1$ belongs $\mathbb P$-a.s. to the set  $\{\ell_k\mid  k \geq 1\}$ and that the sequence $(S_{\ell_{k}})_{k \geq 1}$ is   increasing, we may write 
\begin{align*}
    \mathcal{C}(x, y) &= \sum_{ k \geq 1} \mathbb{P} [x+ S_{\ell_{k-1}}\leq -1,  x+ S_{\ell_{k}}\ = y] \\
     &= \displaystyle \sum_{k \geq 1} \displaystyle \sum_{t = 0}^{-x-1} \mathbb{P}[S_{\ell_{k-1}} = t] \, \mathbb{P}[S_{\ell_k}-S_{\ell_{k-1}} = y - x - t]\\
    &= \displaystyle \sum_{t =0}^{-x-1} \mathbb{P}[S_{\ell_1} = y - x - t]\, \displaystyle \sum_{i\geq 0}  \mathbb{P}[S_{\ell_i} = t]\\
    &= \displaystyle \sum_{t= 0}^{-x-1} \mu_+(y - x - t) \,  \mathcal{U}_+(t).
\end{align*} 
\end{proof}   
    When {\bf H2} holds, the crossing sub-process $\mathcal X_{\bf C}$ is well defined and it  is irreducible, aperiodic  and positive recurrent on its  unique essential class $\mathcal{I}_{\bf C}(X_0)$.  Notice that this essential class  can be a  proper subset of $\mathbb Z$; it occurs  for instance when the support of $\mu$ is bounded from above or the one of $\mu'$ is bounded from below. Nevertheless,  it admits a unique invariant probability measure $\nu$ supported by $\mathcal{I}_{\bf C}(X_0)$. Also in \cite{Vo}, the explicit expression of $\nu$ is  only known when $\alpha \in \{0, 1\}$ and the support of $\mu$ (resp. $\mu'$)  is included in $\mathbb Z^+$  (resp.  in  $ \mathbb{Z^-}$). However, the existence of $\nu$ is enough for our purpose regardless of its exact formula. 

Furthermore, by Theorem $2.2$ in \cite{Vo},  under hypothesis {\bf H3}, the oscillating random walk $\mathcal X$ is irreducible on $\mathbb Z$. It is also important to notice the following property. 
 \begin{property} \label{zaiugegfh}
 Under hypotheses {\bf H1}-- {\bf H3}, the oscillating random walk is null recurrent.  In other words, setting ${\bf t}_0:=\{k\geq 1 \mid  X_k=0\}$ then it holds
 $$\mathbb P_0[{\bf t}_0<+\infty]=1 \quad \text{ and } \quad \mathbb E_0[{\bf t}_0]=+\infty.$$ 
  \end{property}
\begin{proof}
    We may choose  $x_0, y_0\geq 1$ s.t. $\mu(-x_0)>0, \mu'(y_0)>0$ and write for any $n \geq 1$,  
  \begin{align*}
  \mathbb P_0[{\bf t}_0>n]&
  \geq  \mathbb P_0[{\bf t}_0 >n, \xi_1=-x_0]+\mathbb P_0[{\bf t}_0 >n, \xi'_1=y_0]
  \\
  &\geq  \alpha\mu(-x_0)  \mathbb P[ \ell_1 >n-1]+(1-\alpha)\mu'(y_0)  \mathbb P [ \ell'_1 >n-1].
  \end{align*}
 Hence $\mathbb E_0[{\bf t}_0]\geq  \alpha\mu(-x_0) (1+ \mathbb E[\ell_1])+ (1-\alpha)\mu'(y_0) (1+ \mathbb E[\ell'_1])=+\infty.$  
\end{proof}

 Another point to insist on here is that any excursion between two consecutive crossing times is uniquely governed by $S$ or $S'$;  thus, all the results obtained in the previous section can be applied. The decomposition technique that exploits this fact is classical and extremely efficient in controlling the varying excursions over time of Markov processes; for example, we use it in the last section to estimate the convergence of finite dimensional distribution. As a direct application, we can prove that the strong law of large numbers still holds for the chain $\mathcal{X}$.

 \begin{lem}
   Assume that $\mathbb{E}[|\xi_n|]+ \mathbb{E}[|\xi'_n|]< +\infty$   and $\mathbb{E}[\xi_n]= \mathbb{E}[\xi'_n]=0$.  Then, it holds
     $$\displaystyle \lim_{n \to +\infty}\frac{X_n}{n}=0 \,\,\,\,\mathbb{P}\text{-a.s.}$$ 
 \end{lem}
\begin{proof}
   We decompose $X_n$ as $X_n= X_n \mathbb{1}_{\{X_n \geq 1\}} + X_n \mathbb{1}_{\{X_n \leq -1\}}$. 
   
   Let us   estimate the first term.   For any $n \geq 1$, there exists a random integer $k(n) \geq 0$ such that $C_{k(n)} \leq n < C_{k(n)+1}$; notice that the condition 
   $X_n \geq 1$ yields $X_{C_{k(n)}} \geq 1$. Hence, we get
   \begin{align*}
      0 \leq \frac{X_n \mathbb{1}_{\{X_n \geq 1\}}}{n} &=\frac{X_{C_{k(n)}}+ S'_n-S'_{C_{k(n)}}}{n} \\
      &\leq\frac{\max\{X_0, \xi_{C_{k(n)}}\}}{n} +\frac{S'_n}{n}-\frac{S'_{C_{k(n)}}}{C_{k(n)}}\frac{C_{k(n)}}{n}\\
      &\leq\frac{\max\{X_0, \xi_{C_{k(n)}}\}}{n} + \frac{S'_n}{n} + \left\vert\frac{S'_{C_{k(n)}}}{C_{k(n)}} \right\vert.
    \end{align*}
   By the strong law of large numbers,  the different terms on the right-hand side above converges $\mathbb{P}\text{-a.s.}$ to $0$; so does $\frac{X_n \mathbb{1}_{\{X_n \geq 1\}}}{n}$. The second term is treated in the same way.
\end{proof}

\subsection{On aperiodic renewal sequences of operators}

Let 
 $(\mathbb{Z}^{\otimes \mathbb{N}}, (\mathcal{P}(\mathbb{Z}))^{\otimes \mathbb{N}}, \mathcal{X}, (\mathbb{P}_x)_{x \in \mathbb{Z}}, \theta)$ be the canonical space, i.e. the space of trajectories associated with the Markov chain  $\mathcal{X}$. For any $x \in \mathbb{Z}$, we probability measure  $\mathbb{P}_x$ is the conditional probability with respect to the event $[X_0=x]$, we denote by  $\mathbb{E}_x$ the corresponding conditional expectation. The operator $\theta$ is the classical shift transformation defined by  $\theta((x_k)_{k\geq 0})= (x_{k+1})_{k\geq 0} $ for any $(x_k)_{k\geq 0} \in \mathbb{Z}^{\otimes \mathbb{N}}$.

In this section, we study the behavior as $n \to +\infty$ of the sequence  
 $$
{H_n(x,y)}=  \sum_{k=1}^{+\infty}\mathbb{P}_x[C_k=n, X_n=y],
 $$ for any $x, y \in \mathbb Z$. Since the position at time $C_k$ may vary, so that the excursions of $\mathcal{X}$ between two successive crossing times  are not independent, it thus motivates us to take into account the long-term behaviours of these quantities and express them in terms of operators related to the crossing sub-process $\mathcal X_{\bf C}.$
For this purpose, we apply a general renewal theorem due to S. Gou{\"e}zel \cite{Gouezel}. This theorem relies on the  decomposition  of the operator $
 \mathcal C$ using  a sequence of operators  $(\mathcal C_n)_{n \geq 1}$  acting on some Banach space  and   that are not so difficult to deal with.

 It is natural in our context to deal with  the operators $\mathcal C_n= (\mathcal C_n(x, y))_{x, y \in \mathbb Z}, n \geq 1, $ defined by: for any $x, y \in \mathbb Z$ and any $n \geq 1$,
 $$\mathcal C_n(x,y):= \mathbb{P}_x[ C_1=n, X_n=y].$$
 
 The relation $\mathcal{C}(x,y)= \sum_{n \geq 1} \mathcal C_n(x, y)$ is straightforward. We also pay attention to the case $x=0$, that is $\mathcal C_1(0, y)= \mathbb{P}_0[X_1=y]= \alpha \mu(y) + (1-\alpha) \mu'(y)$ and $\mathcal C_n(0, y)=0$ if $n \geq 2$. 
 
 For a function $ \varphi: \mathbb Z \to \mathbb C$,  we formally set
 $$\mathcal C_n\varphi(x) := \displaystyle \sum_{y\in \mathbb{Z}:\, xy\leq 0} \mathcal C_n(x,y)\varphi(y)= \mathbb{E}_x[\varphi(X_n),  C_1=n] \quad \text{if } x \in \mathbb{Z}\setminus\{0\},$$
 and $\mathcal C_1 \varphi(0)= \displaystyle \sum_{y \in \mathbb{Z}}\mathcal C_1(0, y) \varphi(y)= \mathbb{E}_0[\varphi(X_1)]$ and $ \mathcal C_n \varphi(0)=0 \quad \text{if } n \geq 2 $. 
The quantity $\mathcal C_n\varphi(x) $ is well defined for instance when $\varphi \in   L^\infty(\mathbb{Z})$. Other  Banach spaces can be  considered; under  moment  assumptions, we describe below the action of  the $\mathcal C_n$ on a bigger Banach space ${\mathcal B_\delta}$,  more suitable to the situation as explained a little further on.

 Notice  that $\mathcal C_n(x, y)= \mathcal C_n \mathbb{1}_{\{y\}}(x)$ for any $x, y \in \mathbb Z$, which yields, by induction,  
\begin{align}\label{sigman}
    {H_n(x,y)} &=\displaystyle \sum_{k=1}^{+\infty} \mathbb{P}_x[C_k=n, X_n=y]\notag\\
    &=\displaystyle \sum_{k=1}^{+\infty} \displaystyle \sum_{j_1+\ldots+j_k=n}\mathbb{P}_x[  C_1=j_1, C_2- C_1=j_2, \ldots, C_k-C_{k-1}=j_k, X_n=y]\notag\\
    &=\displaystyle \sum_{k=1}^{+\infty} \displaystyle \sum_{j_1+\ldots+j_k=n}{\mathcal C}_{j_1}\ldots {\mathcal C}_{j_k} 1_{\{y\}}(x).
\end{align}

 As announced above, we apply a result of S. Gouezel, stated in a general framework \cite{Gouezel}, that of {\it aperiodic  renewal sequence of operator}, i.e. sequences  $(\mathcal C_n)_{n \geq 1}$  of operators acting on a  Banach space $({\mathcal B}, \vert \cdot \vert_{\mathcal B })$ and satisfying the following conditions: 

 $\bullet$ the operators $\mathcal C_n,  n \geq 1$, act  on ${\mathcal B }$  and $\displaystyle \sum_{n \geq 1} \Vert \mathcal  C_n\Vert_{\mathcal B } <+\infty$ (where $\Vert\cdot \Vert_{\mathcal B }$ denotes the norm on the space $\mathcal L({\mathcal B })$ of continuous operators on $({\mathcal B },  \vert \cdot \vert _{\mathcal B }$));

 $\bullet$ the operator   $\displaystyle \mathcal C(z):= \sum_{n \geq 1} z^n \mathcal C_n$,   defined for  any $z \in \overline{\mathbb{D}}$,  satisfies  

 \begin{enumerate}[R1-]
\item   $\mathcal C(1)$ has a simple eigenvalue at 1 (with corresponding eigenprojector $\Pi$) and the rest of its spectrum is contained in a disk of radius $<1$;

\item for any complex number $z \in \overline{\mathbb{D}} \setminus \{1\}$, the spectral radius of $\mathcal C(z)$ is $<1$; 

\item  for any $n \geq 1$, the real number $r_n$ defined by $\Pi \mathcal{C}_n \Pi= r_n \Pi$ is $\geq 0$.   
\end{enumerate}

Condition R2 implies that,  for any $ z \in \overline{\mathbb D}\setminus \{1\}$, the operator $I-\mathcal  C(z)$ is invertible on ${\mathcal B }$  and
\[
(I-\mathcal  C(z))^{-1}= \sum_{k \geq 0} \mathcal{C}(z)^k= \sum_{k \geq 0} \left(\sum_{j \geq 1}\mathcal C_jz^j\right)^k=\sum_{n \geq 0}   {\mathcal H}_n z^n
\]
with    $  {\mathcal H} _0= I$ and $\displaystyle   {\mathcal H} _n= \sum_{k=1}^{+\infty} \sum_{j_1+ \ldots + j_k=n}{\mathcal C}_{j_1}\ldots {\mathcal C}_{j_k}.
$ The above identity, called the {\it renewal equation},   is of fundamental importance to understand the asymptotics of $H_n$ in the non-commutative setting;   in particular, the equality (\ref{sigman}) yields $H_n(x, y)=   {\mathcal H}_n \mathbb 1_{\{y\}}(x)$ so that the asymptotic behaviour  of $(H_n(x, y))_{n \geq 1}$ is related to that  of $(\mathcal H_n)_{n \geq 1}$.

By  \cite{Gouezel}, if  the sequence $(\mathcal C_n)_{n \geq 1 }$ satisfies the following additive assumptions

\noindent R4($\ell,\beta).    \quad \displaystyle     \Vert \mathcal C_n\Vert_{{\mathcal B }} \leq C \frac{\ell(n)}{n^{1+\beta}}$,

\noindent R5($\ell,\beta). \quad \displaystyle      \sum_{j>n}r_j \thicksim \frac{\ell(n)}{n^\beta}$,

\noindent where $C>0,\ \beta \in (0, 1)$ and   $\ell$ is a slowly varying function, 
then  the sequence $(n^{1-\beta} \ell(n) {\mathcal H}_n)_{n\geq 1}$ converges in $(\mathcal L({{\mathcal B }}), \Vert \cdot \Vert_{\mathcal B })$ to the operator $d_\beta \Pi$, with $d_\beta = \frac{1}{\pi}\sin \beta \pi$.

 In the next subsection, we introduce some  Banach space  $\mathcal B={\mathcal B_\delta}$  in order to be able to apply this general result.

\subsection{Spectral property of the transition matrix \texorpdfstring{$\mathcal{C}=(\mathcal{C}(x,y))_{x,y \in \mathbb{Z}}$}{C} } 
The operator  $\mathcal{C}$ acts on the space  $L^\infty(\mathbb{Z})$ of bounded functions on $\mathbb{Z}$.  By the following lemma, it satisfies some strong spectral property on this space.
 \begin{lem} \label{Doeblin} Assume {\bf H1}-- {\bf H3} hold. Then, the infinite matrix $\mathcal{C}$ satisfies the Doeblin condition and therefore, it is a quasi-compact operator on $L^\infty(\mathbb{Z})$, the space of bounded functions on $\mathbb{Z}$. Furthermore, the eigenvalue
$1$ is simple, with associated eigenvector $\mathbf{1}$, and the rest of the spectrum is included in
a disk of radius $<1$.
\end{lem}
 \begin{proof}
Under the above assumptions,  the positive random variable $S_{\ell_1}$ has finite first moment; hence, by the renewal theorem,  
 $$\displaystyle \lim_{t\to +\infty}  \mathcal{U}_+(t)=\frac{1}{\mathbb{E}[S_{\ell_1}]}>0.$$ 
The above convergence readily implies $\delta_+:=\displaystyle \inf_{z\in \mathbb{Z^+}}{\mathcal U}_+(z)>0.$ 
 
 Consequently, by (\ref{kernel_paper02}), for any $x \leq -1$ and $y\geq 0$,
 $$\mathcal{C}(x,y) \geq \mu_+(y+1)\,  \mathcal{U}_+(-x-1)\geq \delta_+  \mu_+(y+1).$$
 In the same vein, one gets $\mathcal{C}(x,y) \geq \delta' \mu'_-(y-1)$ for any $x\geq 1$ and $y\leq 0$ with $\delta':=\inf_{z \in \mathbb{Z^-}} \mathcal{U'}_-(z)>0$. Hence, it is easy to show that there exists a probability measure $\mathbf{m}$ and $\delta_0>0$ s.t. for any $x\in \mathbb Z$,
$$\mathcal{C}(x,.)\geq \delta_0 \mathbf{m}(.),$$
which immediately implies the quasi-compactness of $\mathcal{C}$. The control of the peripheral spectrum   readily follows.
  \end{proof}
 

Thanks to this lemma, we could believe that  hypothesis R1 is satisfied by  the sequence $(\mathcal C_n)_{n \geq 1}$  acting on $L^\infty(\mathbb Z)$ since $\mathcal C(1) = \mathcal C$. Unfortunately, it holds $\displaystyle \sum_{n \geq 1} \vert \mathcal C_n\vert_\infty =+\infty$. Indeed,  it holds $\vert \mathcal C_n\vert_\infty = \displaystyle \sup_{x \in \mathbb Z} \mathbb P_x[C_1=n]$; now, if we assume for instance $x \leq   1$, it holds $\mathbb P_{x}[C_1=n]= \mathbb P[\tau^S(x)=n] $ with 
 
 $(i)  \quad \mathbb P[\tau^S(x)=n] = O(1/n),\ \quad $  
 
\noindent  and
 
  $(ii) \quad \displaystyle \liminf_{n \to +\infty} n\mathbb P[\tau^S(x_n)=n] >0$  when $\  x_n\asymp \sqrt{n}\ 
.$

\noindent (see  Lemma 5 and Theorem (B) \cite{Eppel}).
  Consequently $\vert \mathcal C_n\vert_\infty \asymp 1/n$.

 Thus, we have to choose  another Banach space ${\mathcal B_\delta}$.
  By (\ref{crossing_time}), it is clear that $C_{k+1}= \tau_S(X_{C_k})$ when  $X_{C_k}\leq -1$ and $C_{k+1}= \tau_{S'}(X_{C_k})$ when  $X_{C_k}\geq  1$. Consequently,  the  behaviour as $n \to +\infty$ of  the  $k^{\rm th}$-term $\mathbb{P}_x[C_k=n, X_n=y]$ of the sum $\Sigma_n(x,y)$  is  closely related to the distributions of $\tau_S$ and $\tau_{S'}$; in particular, by Lemma \ref{equal}, its dependence on   $y$ is  expressed in terms of $h_a(y)$ and $h'_d(y)$. This explains why  we have to choose a Banach space on which the action of $\mathcal C$ has ``nice'' spectral properties - as compacity or quasi-compacity - and also    does contain these functions $h_a$ and $h'_d$. The fact that they   are both sublinear leads us  to examine the action of $\mathcal C$ on  the space ${\mathcal B_\delta}$ of complex valued functions on $\mathbb Z$ defined by  

$$  {\mathcal B_\delta} :=\left\{f:\mathbb{Z} \to \mathbb{C}: \vert f \vert_{\mathcal B_\delta}:=\sup_{x\in \mathbb{Z}}\frac{|f(x)|}{1+ |x|^{1+\delta}} <+\infty \right\},$$
with $\delta \geq 0$. 

By Lemma \ref{upperbound} and the fact that $h_a(x)=O(x), \ h'_d(x)=O(x)$, the functions $h_a, h'_d, {\bf h}_n: x\mapsto \sqrt{n} \mathbb P[\tau^S(-x)>n] $ and  ${\bf h}'_n: x\mapsto \sqrt{n} \mathbb P[\tau^{S'}(x)>n]$ do belong to ${\mathcal B_\delta}$ for any $\delta\geq 0$;  furthermore, applying Lemma \ref{equal}, the sequence $({\bf h}_n)_{n \geq 0}$ (resp. $( {\bf h}'_n )_{n \geq 0}$) converges to $2 c h_a$ (resp. $2 c'h'_d$) in ${\mathcal B_\delta}$ if  $\delta>0$. This last property is of interest  in applying Gou\"ezel's renewal theorem and  for this reason,  we assume from now on  $\delta >0$.

 Furthermore, the map  $\mathcal{C}$  acts on ${\mathcal B_\delta}$ as a compact operator whose spectrum   can be controlled as follows.

\begin{prop} \label{first} Assume that hypotheses {\bf H1}-- {\bf H4} hold. Then,     
\begin{enumerate}[(i) ]
    \item The map $\mathcal{C}$ acts on ${{\mathcal B_\delta}}$ and $\mathcal{C}({{\mathcal B_\delta}}) \subset L^\infty(\mathbb{Z})$.
    \item  $\mathcal{C}$ is a compact operator on ${{\mathcal B_\delta}}$ with spectral radius $\rho_{\mathcal B_\delta}  =1$ and with the unique and simple dominant eigenvalue $1$. 
    \item The rest of  the spectrum of $\mathcal C$   on ${\mathcal B_\delta}$ is contained in a disk of radius $<1$.
\end{enumerate}
 Consequently, the operator $\mathcal C$ on ${\mathcal B_\delta}$ may be decomposed as
\begin{align*} 
\mathcal C= \Pi + Q
\end{align*}
where 
\\
\indent $\bullet$ \ \ $\Pi$ is the eigenprojector from ${\mathcal B_\delta}$ to $\mathbb C \bf{1}$ corresponding to the eigenvalue $1$ and $\Pi(\phi)= \nu(\phi) \mathbf{1}$, where $\nu$ is the unique $\mathcal{C}$-invariant probability measure on $\mathbb Z$;

\indent $\bullet$ \ \ the spectral radius of $Q$  on ${\mathcal B_\delta}$ is $<1$; 

 $\bullet$ \ \ $\Pi Q= Q\Pi = 0$.
\end{prop}
\begin{proof}
 \begin{enumerate}[i)]
     \item Note that $ \mathcal{U'}_-(t)=\displaystyle \sum_{n\geq 0} \mathbb{P}[S'_{\ell'_n}=t] = \mathbb{P}[\exists n\geq 0: S'_{\ell'_n}=t] \leq 1$. For any $\varphi \in {{\mathcal B_\delta}}$ and $x\geq 1$, we have  
     \begin{align*}
         |\mathcal{C}\varphi(x)| &\leq \displaystyle \sum_{y\leq 0} \displaystyle \sum_{t = -x+1}^{0} \mu'_-(y-x-t)\, \vert  \varphi(y)|\\
         &\leq  \vert  \varphi \vert_{\mathcal B_\delta}\displaystyle \sum_{y\leq 0}(1+\vert y\vert^{1+\delta} )\, \mu'_-(-\infty, y-1)\\ 
         &\leq \vert  \varphi \vert_{\mathcal B_\delta} \left(\mathbb{E}[\vert S'_{\ell'_1}] + \mathbb{E}[\vert S'_{\ell'_1} \vert^{2+\delta}]\right),
     \end{align*}
     which is finite if  $\mathbb{E}[(\xi'^-_n)^{3+\delta}]<+\infty$ (see \cite{ChowLai}). Other cases can be estimated in the same way and  yield
     \begin{align} \label{bound}
         |\mathcal{C}\varphi \vert_{\mathcal B_\delta} \leq \vert \mathcal{C}\varphi\vert _\infty  \leq  \vert  \varphi \vert_{\mathcal B_\delta} \left(\mathbb{E}[\vert S'_{\ell'_1}] +\mathbb{E}[\vert S'_{\ell'_1} \vert^{2+\delta}]\right) <+\infty.
     \end{align} 
     \item By \eqref{bound}, the operator $\mathcal{C}$ acts continuously from  ${{\mathcal B_\delta}}$ into $L^\infty(\mathbb{Z})$; since the inclusion map $i: L^\infty(\mathbb{Z}) \hookrightarrow {{\mathcal B_\delta}}$ is compact, the operator $\mathcal{C}$ is also compact on ${{\mathcal B_\delta}}$. 
     
        Let us now compute the spectral radius $ \rho_{\mathcal B_\delta}$ of $\mathcal{C}$. The fact that $\mathcal{C}$ is a stochastic matrix yields $ \rho_{\mathcal B_\delta} \geq 1$. To prove $ \rho_{\mathcal B_\delta}\leq 1$,  it suffices to show that $\mathcal{C}$ has bounded powers on ${{\mathcal B_\delta}}$. For any $n \geq 1$ and $x \in \mathbb{Z}$,
    \begin{align*}
        \vert \mathcal{C}^{n}\varphi(x) \vert \leq \displaystyle \sum_{y\in \mathbb{Z}} \mathcal{C}^{n-1}(x,y) \vert \mathcal{C}\varphi(y) \vert \leq \vert \mathcal{C}\varphi \vert_{\infty} \displaystyle \sum_{y\in \mathbb{Z}} \mathcal{C}^{n-1}(x,y) =\vert \mathcal{C}\varphi \vert_{\infty}.
    \end{align*}
Together with \eqref{bound}, it implies
    \begin{align*}
        \vert \mathcal{C}^{n} \varphi \vert_{\mathcal B_\delta}  \leq \vert \mathcal{C}^n \varphi \vert _\infty \leq |\mathcal{C}\varphi|_\infty \leq \vert  \varphi\vert _{\mathcal B_\delta}\left(\mathbb{E}[\vert S'_{\ell'_1}] +\frac{1}{2}\mathbb{E}[\vert S'_{\ell'_1} \vert^{2+\delta}]\right).
    \end{align*}
Hence $\Vert \mathcal C^n\Vert_{\mathcal B_\delta}\leq \mathbb{E}[\vert S'_{\ell'_1}] + \mathbb{E}[\vert S'_{\ell'_1} \vert^{2+\delta}]$ for any $n \geq 1$ and  
 $ \rho_{\mathcal B_\delta}= \displaystyle \lim_{n\to +\infty} \Vert \mathcal{C}^n \Vert_{\mathcal B_\delta}^{1/n}  \leq 1.$

Let us now control the peripheral spectrum of $\mathcal C$. Let $\theta \in \mathbb{R}$ and $\psi \in {{\mathcal B_\delta}}$ such that $\mathcal{C}\psi = e^{i\theta}\psi$. Obviously, the function $\psi$ is bounded and $\vert \psi \vert \leq \mathcal{C} \vert \psi \vert$.  Consequently,  $\vert \psi \vert_{\infty} -\vert \psi \vert$ is  non-negative and  super-harmonic    (i.e. $\mathcal{C}(\vert \psi \vert_{\infty} -\vert \psi \vert) \leq \vert \psi \vert_{\infty} -\vert \psi \vert$)  on the unique irreducible class  $\mathcal{I}_{\bf C}(X_0)$ of $\mathcal X$. By the classical denumerable Markov chains theory, it is thus constant on $\mathcal{I}_{\bf C}(X_0)$ which follows that $\vert \psi\vert$ is constant on $\mathcal{I}_{\bf C}(X_0)$.   

 \quad Without loss of generality, we may assume $\vert \psi(x) \vert =1$ for any $x \in \mathcal{I}_{\bf C}(X_0)$, i.e.  $\psi(x) = e^{i\phi(x)}$ for some $\phi(x) \in \mathbb{R}$. We may rewrite the equality $\mathcal{C}\psi = e^{i\theta}\psi$ as 
    \begin{align*}
     \forall x \in \mathcal{I}_{\bf C}(X_0) \quad  \sum_{y \in {\mathcal{I}_{\bf C}}(X_0)} \mathcal{C}(x,y) e^{i(\phi(y) - \phi(x))} = e^{i\theta}.
    \end{align*}
   Note that $\mathcal{C}(x,y) >0$ for all $x, y \in \mathcal{I}_{\bf C}(X_0)$;   by convexity, one readily gets $e^{i\theta} =e^{i(\phi(y) - \phi(x))}$ for such points $x, y$. Taking $x = y\in   \mathcal{I}_{\bf C}(X_0)$, we thus  obtain $e^{i\theta}=1$.

   In particular, the function $\psi$ is harmonic on $\mathcal{I}_{\bf C}(X_0)$, hence constant on this set, by \textit{Liouville's theorem}. Furthermore,  for any $x \in \mathbb Z$, it holds   $\mathcal{C}(x,y)>0 \Longleftrightarrow y \in \mathcal{I}_{\bf C}(X_0)$ ; consequently, for  any  fixed $y_0    \in \mathcal{I}_{\bf C}(X_0)$ and   any $x \in \mathbb Z$, 
   \begin{align*}
       \psi(x)=\mathcal{C}\psi(x)= \displaystyle \sum_{y \in {\mathcal{I}_{\bf C}}(X_0)} \mathcal{C}(x,y)\psi(y)= \psi(y_0). 
   \end{align*}
Therefore, the function $\psi$ is constant on $\mathbb{Z}$.     
   \item This is a direct consequence of (ii). 
 \end{enumerate}
\end{proof} 

 \subsection{A renewal limit theorem for the sequence of crossing times}
The main goal of this part is to prove the following statement.
 \begin{prop} \label{convergeH} 
 The sequence $(\sqrt{n} \mathcal H_n)_{n \geq 1}$ converges in $(\mathcal L({\mathcal B_\delta}), \Vert \cdot \Vert_{\mathcal B_\delta})$ to   the operator ${\bf c}^{-1}\Pi$ with 
 $\displaystyle  {\bf c}={2\pi\bigl(c\nu( \check{h}_a)+c'\nu(h'_d)\bigr)}$. In particular, 
for any $x, y \in \mathbb Z$,
 \[\lim_{n \to +\infty} \sqrt{n} H_n(x, y)=\frac{\nu(y)}{2\pi\bigl(c\nu( \check{h}_a)+c'\nu(h'_d)\bigr)}.
 \]
\end{prop} 
This is a consequence of the fact that  $(\mathcal C_n)_{n \geq 1} $ is an aperiodic renewal sequence of operators on ${\mathcal B_\delta}$ satisfying  R4  and  R5  ( with $\beta = 1/2$ and $\ell$ constant). 
 
The fact that all the $\mathcal C_n, n \geq 1$, act on ${\mathcal B_\delta}$  and  $\sum_{n \geq 1} \Vert \mathcal C_n\Vert_{\mathcal B_\delta} <+\infty$ is a consequence of the following lemma. 
 \begin{lem} \label{oper_bound}
 Under hypotheses {\bf H1}-- {\bf H4}, for any $n\geq 1,$ the operator $\mathcal C_n$ acts on ${\mathcal B_\delta}$ and 
 $$\Vert \mathcal C_n \Vert_{\mathcal B_\delta} = O\left(\dfrac{1}{n^{3/2}}\right).$$
 \end{lem}
 \begin{proof}
By \eqref{eventS}, for any $x \geq 1$ and $\phi \in {{\mathcal B_\delta}}$, 
\begin{align*}
 \vert \mathcal C_n\phi(-x) \vert  &\leq  \sum_{w \geq 0} \mathbb{P}_{-x}[ C_1=n; X_n=w]\, |\phi(w)|\\ 
      &= \displaystyle \sum_{w \geq 0}\mathbb{P}[\tau^S(-x)=n, -x+S_n=w]\, |\phi(w)|\\
      &\preceq \frac{1+x}{n^{3/2}}  \sum_{w \geq 0} \left(\sum_{z \geq w+1} z \mu(z)\right)  |\phi(w)|\\
      &\leq \frac{1+x}{n^{3/2}}\vert \phi\vert_{\mathcal B_\delta} \underbrace{\sum_{w \geq 0} (1+w^{1+\delta})\left(\sum_{z \geq w+1} z \mu(z)\right)}_{\preceq  \displaystyle \sum_{z \geq 1}z^{3+\delta}\mu(z)= \mathbb E[(\xi_1^+)^{3+\delta}]}.
  \end{align*}
Similarly, 
$$\displaystyle \vert \mathcal C_n\phi(x) \vert  \preceq  \frac{1+x}{n^{3/2}}\vert \phi\vert_{\mathcal B_\delta}\mathbb{E}[(\xi'^-_1)^{3+\delta}].$$  
Moreover, $|\mathcal C_1 \phi(0)| \preceq  \vert \phi \vert_{\mathcal B_\delta}$  and $\mathcal C_n\phi(0)=0$ for all $n\geq 2$. This completes the proof.
 \end{proof}
  Condition R1   coincides with the statement of Proposition \ref{first}. Similarly,  R2 and R3 correspond to assertions $i)$ and $ii)$ of the next proposition. Consequently, $(\mathcal C_n)_{n \geq 1}$ is an aperiodic renewal sequence of operators.
\begin{prop} \label{R2345}Suppose that {\bf H1}-- {\bf H4} are satisfied. Then the sequence $(\mathcal C_n)_{n\geq 1}$ holds the following properties
\begin{enumerate}[i)]
      \item The spectral radius $\rho_{\mathcal B_\delta}(z)$ of $\mathcal{C}(z)$ is strictly less than $1$ for $z \in \overline{\mathbb{D}}\setminus\{1\}$. \\
  \item For any $n\geq 1$, it holds $\Pi \mathcal C_n \Pi = r_n \Pi $ with  
    \[
    r_n:=\nu(\mathcal C_n 1)=\displaystyle \sum_{x\in \mathbb{Z}}\nu(x)\mathbb{P}_x[\mathcal \mathcal  C_1=n]\geq 0.
    \]
    \item $\displaystyle \sum_{j>n}r_j \thicksim\frac{2\bigl(c\nu( \check{h}_a)+c'\nu(h'_d)\bigr)}{\sqrt{n}}$ as $ n \to +\infty$. 
\end{enumerate}
\end{prop}
\begin{proof}
 \begin{enumerate}[$i)$]
     \item The argument  is close to the one used to prove  Proposition \ref{first}. For any $z\in \overline{\mathbb{D}}\setminus\{1\}$, the operator $\mathcal{C}(z)$ is compact on ${{\mathcal B_\delta}}$ with spectral radius $ \rho_{\mathcal B_\delta}(z)\leq 1$. We now prove $ \rho_{\mathcal B_\delta}(z) \neq 1$ by contraposition.  Suppose $\rho_{{{\mathcal B_\delta}}}(z)=1$; in other words, there exist $\theta \in \mathbb{R}$ and $\varphi \in {{\mathcal B_\delta}}$ such that $\mathcal{C}(z)\varphi= e^{i\theta} \varphi$. Since $\mathcal{C}$ is bounded from ${{\mathcal B_\delta}}$ into $L^\infty(\mathbb{Z})$ and $0\leq \vert  \varphi| \leq \mathcal{C}\vert  \varphi|$, the function $\vert  \varphi|$ is $\mathcal C$- superharmonic, bounded and thus constant on its essential class $\mathcal{I}_{\bf C}(X_0)$. 
     
 Without loss of generality, we can suppose  that $\vert  \varphi(x)|=1$ for any  $x \in \mathcal{I}_{\bf C}(X_0)$; equivalently, $\varphi(x)=e^{i\phi(x)}$  for  some function  $\phi:\mathcal{I}_{\bf C}(X_0) \to \mathbb R$. For any $x \in \mathcal{I}_{\bf C}(X_0)$, we get
     \begin{align*}
         \mathcal{C}(z)\varphi(x)= e^{i\theta}\varphi(x) &\Longleftrightarrow \displaystyle \sum_{n\geq 1} \displaystyle \sum_{y \in \mathcal{I}_{\bf C}(X_0)} z^n e^{i(\phi(y)-\phi(x))} \mathbb{P}_x[C_1=n; X_n=y]=e^{i\theta},
     \end{align*}
with $\displaystyle \sum_{n\geq 1} \displaystyle \sum_{y \in \mathcal{I}_{\bf C}(X_0)} \mathbb{P}_x[C_1=n; X_n=y]=1$. By convexity, it readily holds   $z^n e^{i(\phi(y)-\phi(x))}=e^{i\theta}$ for all $x, y \in \mathcal{I}_{\bf C}(X_0)$ and $n \geq 1$.
  By taking $x=y$, we obtain $z^n=e^{i\theta}$ for all $n\geq 1$; consequently  $z=1$,  contradiction.
\item For any $\phi \in {{\mathcal B_\delta}}$ and $n\geq 1$,
 $$
         \Pi \mathcal C_n \Pi \phi= \nu(\phi)\Pi(\mathcal C_n \mathbf{1})  = \nu(\phi) \mathbf{1} \displaystyle \sum_{z\in \mathbb{Z}} \nu(z) \mathbb{P}_z[C_1=n] = r_n\Pi(\phi) 
 $$
     with $r_n=  \displaystyle \sum_{z\in \mathbb{Z}} \nu(z) \mathbb{P}_z[C_1=n]\geq 0$.
    \item On the one hand, by Proposition \ref{first}, the eigenprojector $\Pi$ acts on ${\mathcal B_\delta}$; thus, since $\check{h} _a \in {\mathcal B_\delta}$, it holds     $\nu(\check{h}_a) <+\infty$. On the other hand, the support $\mathcal{I}_{\bf C}(X_0)$ of $\nu$ intersects $\mathbb Z^-$ and the support of $\check{h}_a$ equals $\mathbb Z^-$;  hence $\nu(\check{h}_a)>0$.   Similarly $0<\nu(h_d')<+\infty$.  

     Now, let us write
     \begin{align*}
        \sum_{j>n}r_j &= \displaystyle \sum_{j>n}  \displaystyle \sum_{x\in \mathbb{Z}} \nu(x) \mathbb{P}_x[\mathcal \mathcal C_1=j]\\
         &= \displaystyle \sum_{x\in \mathbb{Z}} \nu(x) \mathbb{P}_x[\mathcal \mathcal C_1>n]\\
         &=\displaystyle \sum_{x \leq -1} \nu(x) \mathbb{P}[\tau^S(x)>n]+\displaystyle \sum_{x \geq 1} \nu(x) \mathbb{P}[\tau^{S'}(x)>n] \quad (\text{since } \mathbb{P}_0[\mathcal \mathcal C_1=1]=1)\\
         &\thicksim\frac{2c}{\sqrt{n}} \displaystyle \sum_{x \leq -1}\nu(x) h_a(-x) +\frac{2c'}{\sqrt{n}} \displaystyle \sum_{x \geq 1}\nu(x) h'_d(x)=2\ \frac{c\  \nu( \check{h}_a) + c' \ \nu(h'_d)}{\sqrt{n}}.
     \end{align*} 
 \end{enumerate} 
\end{proof}
Finally, combining Lemma  \ref{oper_bound} and Proposition  \ref{R2345} $iii)$, we see that conditions R4  and  R5 are satisfied with $\ell=  {\rm const}=  2 (c\  \nu( \check{h}_a) + c' \ \nu(h'_d))$  and $\beta = 1/2$.

Consequently, by \cite{Gouezel}, 
the sequence  $(\sqrt{n}\mathcal H_n)_{n\geq 1}$ converges in $(\mathcal L({{\mathcal B_\delta}}), \Vert \cdot \Vert_{\mathcal B_\delta})$ to the operator ${\bf c}^{-1}\Pi$ with 
 $\displaystyle  {\bf c}={2\pi\bigl(c\nu( \check{h}_a)+c'\nu(h'_d)\bigr)}$. Formally, one may write
\begin{align} \label{consequence}
\bigg\Vert \sqrt{[ns]}{\mathcal H}_{[ns]}- {\bf c}^{-1}\Pi \bigg\Vert_{\mathcal B_\delta} \longrightarrow 0, \quad \text{as} \quad n \to +\infty.
\end{align}
\section{Proof of  Theorem \ref{main_paper02}} 

For $m \geq 1$, let $\{\varphi_i: \mathbb{R} \to \mathbb{R} \mid i= 1,\ldots, m\}$ be a sequence of bounded and Lipschitz continuous functions with corresponding Lipschitz coefficients $Lip(\varphi_i)$. Assume that the time sequence  $\{t_i\}_{1\leq i \leq m}$ is strictly increasing with values in $(0,1]$ and $t_0=0$. In this part, we prove that
$$\displaystyle \lim_{n\to +\infty} \mathbb{E}_x \left[\prod_{i=1}^m \varphi_i \bigg(X^{(n)}(t_i)\bigg)\right]= \displaystyle \int_{\mathbb{R}^m} \prod_{i=1}^m \varphi_i(u_i) p^{\gamma}_{t_i-t_{i-1}}(u_{i-1}, u_i) \, du_1 \ldots du_m $$
with $u_0=0$.
 
 Without loss of generality, we assume $\sigma=\sigma'$ and $x \geq 1$ to reduce unnecessary complexity associated with subcases.  
 
\subsection{Convergence of the one dimensional distributions \texorpdfstring{$m=1$}{m}}

 We first notice that $\mathbb{E}_x[\varphi_1(X^{(n)}(t_1))] \approx \mathbb{E}_x\left[\varphi_1\left(\dfrac{X_{[nt_1]}}{\sigma\sqrt{n}}\right)\right]$ since
\begin{align*}
\bigg\vert \mathbb{E}_x[\varphi_1(X^{(n)}(t_1))]-\mathbb{E}_x\left[\varphi_1\left(\dfrac{X_{[nt_1]}}{\sigma\sqrt{n}}\right)\right] \bigg\vert &\leq Lip(\varphi_1)\, \mathbb{E}_x\left[\bigg\vert X^{(n)}(t_1)-\dfrac{X_{[nt_1]}}{\sigma\sqrt{n}}\bigg\vert\right]\\
&\leq Lip(\varphi_1)\,\dfrac{\mathbb{E}[|\xi_{[nt_1]+1}|]+\mathbb{E}[|\eta_{[nt_1]+1}|]+\mathbb{E}[|\xi'_{[nt_1]+1}|]}{\sigma \sqrt{n}},  
\end{align*}
which tends to $0$ as $n\to +\infty$.
 Now, we can decompose $\mathbb{E}_x\left[ \varphi_1\left(\dfrac{X_{[nt_1]}}{\sigma \sqrt{n}}\right)\right]$ as
$$\underbrace{\mathbb{E}_x\left[\varphi_1\left(\dfrac{X_{[nt_1]}}{\sigma \sqrt{n}}\right), X_{[nt_1]}=0\right]}_{A_0(n)}
+\underbrace{\mathbb{E}_x\left[\varphi_1\left(\dfrac{X_{[nt_1]}}{\sigma \sqrt{n}}\right), X_{[nt_1]}>0\right]}_{A^+(n)}
+ \underbrace{\mathbb{E}_x\left[ \varphi_1\left(\dfrac{X_{[nt_1]}}{\sigma \sqrt{n}}\right), X_{[nt_1]}<0\right]}_{A^-(n)}.$$ 
 The   term $A_0(n)$ tends to $0$ as $n \to +\infty$ since $(X_n)_{n \geq 0}$ is null recurrent. It remains to control the two other terms.
 
\noindent$\bullet$ Estimate of $A^+(n)$
\begin{align*} 
A^+(n)\approx &\displaystyle {\sum_ { k_1=1}^{ [nt_1]-1} \displaystyle \sum_{\ell \geq 1} }\displaystyle \sum_{y\geq 1}\mathbb{E}_x\bigg[ \varphi_1\left(\dfrac{X_{[nt_1]}}{\sigma \sqrt{n}}\right), C_\ell= k_1, X_{ k_1}=y, y +\xi'_{ k_1+1}>0, \\
& \hspace{5cm} \ldots, y+\xi'_{ k_1+1}+\ldots+\xi'_{[nt_1]}>0\bigg]\\
 =\displaystyle &{\sum_ { k_1=1}^{ [nt_1]-1} }\displaystyle \sum_{y \geq 1} \mathbb{E}\left[ \varphi_1\left(\dfrac{y+\xi'_{ k_1+1}+\ldots+\xi'_{[nt_1]}}{\sigma \sqrt{n}}\right), \tau^{S'}(y)>[nt_1]- k_1\right]\\
 & \hspace{6cm} \left(\displaystyle {  \sum_{\ell\geq 1}}\mathbb{P}_x[C_\ell= k_1; X_{ k_1}=y]\right) \\ 
 =\displaystyle &{\sum_{ k_1=1}^{[nt_1]-1}}\displaystyle \sum_{y \geq 1} H_{ k_1}(x,y) \mathbb{E}\left[ \varphi_1\left(\dfrac{y+S'_{[nt_1]- k_1}}{\sigma \sqrt{n}}\right),  \tau^{S'}(y)>[nt_1]- k_1\right].
\end{align*}
    For any $0 \leq s_1 \leq   t_1$ and $n \geq 1$, let $f_n$ be the function defined by
$$f_n(s_1):=n\displaystyle \sum_{y \geq 1} H_{[ns_1]}(x,y) \mathbb{E}\left[ \varphi_1\left(\dfrac{y+S'_{[nt_1]-[ns_1]}}{\sigma \sqrt{n}}\right),  \tau^{S'}(y)>[nt_1]-[ns_1]\right]$$
if $0\leq s_1 < \frac{[nt_1]}{n}$  and $f_n(s_1)=0$ if $\frac{[nt_1]}{n}\leq s_1\leq t_1 $. Hence,  
\begin{align*}
    A^+(n)= \displaystyle \int_{0}^{t_1}f_n(s_1)\,ds_1 + O\left(\dfrac{1}{\sqrt{n}}\right).
\end{align*}
%
%
The convergence of the term $A^+(n)$  as $n\to +\infty$ is a consequence of the two following properties:

\noindent 
$\bullet$ for any $n \geq 1$,
\begin{equation} \label{iezhrkv}
\vert f_n(s_1) \vert \preceq\frac{1+|x|}{\sqrt{s_1(t_1-s_1)}}  \in L^1([0,t_1]).
\end{equation}
$\bullet$ 
for any $s_1 \in [0, t_1]$,   
\begin{align} \label{ekhrf}
\lim_{n\to +\infty} f_n(s_1)&= 
\frac{\gamma}{\pi\sqrt{s_1(t_1-s_1)}}  \int_{0}^{+\infty}\varphi_1(z\sqrt{t_1-s_1})z \exp\left(\frac{-z^2}{2}\right)dz
\notag
\\
&=\dfrac{\gamma}{\pi}  \int_0^{+\infty}\varphi_1(u) u  \ \frac{\exp\left(\dfrac{-u^2}{2(t_1-s_1)}\right)}{\sqrt{s_1(t_1-s_1)^3}}du \quad \quad (\text{set }u=z\sqrt{t_1-s_1}).
\end{align} 
Indeed, applying the Lebesgue dominated convergence theorem, we obtain
\begin{align}\label{positive}
   \lim_{n\to +\infty} A^+(n) 
      &=
   \dfrac{\gamma}{\pi}   \int_0^{+\infty}\varphi_1(u) u\left(  \int_0^{t_1}\frac{1}{\sqrt{s_1(t_1-s_1)^3}}\exp\left(\dfrac{-u^2}{2(t_1-s_1)}\right)ds_1 \right)du \nonumber \\ 
    &=\dfrac{\gamma}{\pi}   \int_0^{+\infty}\varphi_1(u)u \bigg[\frac{1}{t_1} \exp\left(\dfrac{-u^2}{2t_1}\right) \underbrace{  \int_0^{+\infty}\dfrac{1}{\sqrt{s}}\exp\bigg(\dfrac{-u^2}{2t_1}s\bigg)ds}_{=\dfrac{\sqrt{2\pi t_1}}{u}}\bigg]du \nonumber \quad (\text{set } s:=\frac{s_1}{t_1-s_1})\\ 
    &=\gamma   \int_0^{+\infty} \varphi_1(u)\frac{2\exp\left(-u^2/{2t_1}\right)}{\sqrt{2\pi t_1}}du. 
\end{align}
Similarly,
 \begin{align}\label{negative}
   \lim_{n\to +\infty}A^-(n)=(1-\gamma)   \int_{-\infty}^0 \varphi_1(u)\frac{2\exp\left(-u^2/{2t_1}\right)}{\sqrt{2\pi t_1}}du.  
 \end{align} 
Combining  \eqref{positive} and \eqref{negative}, we thus obtain  
   $$  \lim_{n\to +\infty} \mathbb{E}_x[\varphi_1(X^{(n)}_{t_1})]=   \int_{\mathbb{R}}\varphi_1(u) p_{t_1}^{\gamma}(0, u)du=   \int_{\mathbb{R}}\Tilde{\varphi_1}(u)\frac{2\exp(-u^2/2t_1)}{\sqrt{2\pi t_1}}du,$$
 where $\Tilde{\varphi}_1(u)= \gamma \varphi_1(u)\mathbb{1}_{(0,+\infty)}(u)+(1-\gamma) \varphi_1(u)\mathbb{1}_{(-\infty,0)}(u).$
 
It thus remains to establish (\ref{iezhrkv}) and (\ref{ekhrf}).   The first natural idea is to set
$$
\psi_n(y):= \sqrt{n} \mathbb{E}\left[\varphi_1\left(\dfrac{y+S'_{[nt_1]-[ns_1]}}{\sigma\sqrt{n}}\right),  \tau^{S'}(y)>[nt_1]-[ns_1]\right]  
 $$
 and to remark that   $ f_n(s_1)=  \sqrt{n} \mathcal H_{[ns_1]}(\psi_{n})(x)$ with $\psi_n \in \mathcal B_\delta$ .
One can easily check that $(\psi_n)_{n\geq 0}$ converges pointwise to some function $\psi \in \mathcal B_\delta$ but it is much more complicated to prove that this convergence  holds in $\mathcal B_\delta$. This can be done  when $\delta \geq 1$ with a strong moment assumption (namely moments of order $\geq 4$ for $\mu'$) by using a recent result in \cite{LLP}; unfortunately, such a result does not exist  for the  Brownian meander, which is useful in the sequel for convergence of multidimensional  distributions. This forces us to propose another strategy that we now present.

 For this purpose, for any $n \geq 1$ and any fixed $0<s_1< t_1$, we decompose $f_n(s_1)$ as $f_n(s_1)=   \sum_{y \geq 1}a_n(y)b_n(y)$, where  
\begin{align*}
  a_n(y) &:= n \displaystyle H_{[ns_1]}(x, y)\,\mathbb{P}[\tau^{S'}(y)>[nt_1]-[ns_1]], \\
b_n(y) &:= \mathbb{E}\left[\varphi_1\left(\dfrac{y+S'_{[nt_1]-[ns_1]}}{\sigma\sqrt{n}}\right) \mid \tau^{S'}(y)>[nt_1]-[ns_1]\right]  
\end{align*}
and apply the following  classical lemma with $V= \mathbb Z^+$:
\begin{lem} \label{product} Let $V$ be   denumerable    and $(a_n(v))_{v\in V}, (b_n(v))_{v \in V  }$ be real sequences satisfying

 (i) $a_n(v) \geq 0$ for any $n \geq 1, v \in V$ and $ \displaystyle \lim_{n \to +\infty} \displaystyle \sum_{v\in V}a_n(v) = A$,
 
 (ii) for any $\epsilon >0$, there exists a finite set $V_\epsilon \subset V$ s.t. $\sup_{n \geq 1} \sum_{v \notin V_\epsilon} a_n(v) <\epsilon$.

(iii) $\displaystyle \lim_{n \to +\infty} b_n(v)= b$ for any $ v \in V$  and $\displaystyle \sup_{n \geq 1, v \in V} |b_n(v)| <+\infty$.

   \noindent  Then
    $\quad \displaystyle \lim_{n \to +\infty} \displaystyle \sum_{v \in V} a_n(v) b_n(v)= Ab.$
 \end{lem}
 Let us check that these  conditions are satisfied by the families $(a_n(y))_{y \geq 1}, (b_n(y))_{y \geq 1  }$  defined above.

 \underline{ Condition $(i).$} The sum $\displaystyle \sum_{y\geq 1} a_n(y) $ may be written as
\begin{align}\label{sum a_n(y)}
\sum_{y\geq 1} a_n(y) = {1+o(n)\over \sqrt{s_1 (t_1- s_1)}} \ \sqrt{[ns_1]}\mathcal H_{[ns_1]}({\bf h}'_{[nt_1]-[ns_1]})(x).
\end{align}
On the one hand,  the sequence  $(\sqrt{[ns_1]}\mathcal H_{[ns_1]})_{n\geq 1}$ converges in $(\mathcal L({{\mathcal B_\delta}}), \Vert \cdot \Vert_{\mathcal B_\delta})$ to the operator ${1\over 2\pi\bigl(c\nu( \check{h}_a)+c'\nu(h'_d)\bigr)}\Pi$; on the other hand, the sequence $({\bf h}'_{[nt_1]-[ns_1]})_{n \geq 1}$ converges  in ${\mathcal B_\delta} $ to $2  c' h_d'$.
Hence condition $(i)$ holds  with $$A= \dfrac{1}{\sqrt{s_1(t_1-s_1)}} {c' \nu(h'_d)\over \pi\bigl(c\nu( \check{h}_a)+c'\nu(h'_d)\bigr)}=\dfrac{\gamma}{\pi \sqrt{s_1(t_1-s_1)}}.$$

 \underline{Condition $(ii)$.}  Fix $\epsilon >0$.  We want to find $y_\epsilon\geq 1$ s.t. $\displaystyle \sum_{y \geq y_\epsilon} a_n(y)\leq \epsilon$ for any $n \geq 1$. By Lemma \ref{upperbound}, there exists a constant $C_0>0$ s.t. $0\leq  {\bf h}'_k(y)\leq C_0(1+y)$ for any $y, k\geq 1$; hence,   for $y \geq y_\epsilon$, 
 $$0\leq  {\bf h}'_k(y)\leq C_0\left(1+{y^{1+\delta}\over y_\epsilon^\delta}\right)\leq 2C_0{  1+y^{1+\delta}\over y_\epsilon^\delta}.$$
 Consequently the function $ {\bf h}'_k {\bf \mathbb{1}}_{[y_\epsilon, +\infty[ }$ belongs to ${\mathcal B_\delta}$ and $  \vert  {\bf h}'_k {\bf \mathbb{1}}_{[y_\epsilon, +\infty[ }\vert _{\mathcal B_\delta} \leq {2C_0\over y_\epsilon^\delta}$ for any $k \geq 1$.
By \eqref{sum a_n(y)}, it follows
    \begin{align*}
 0\leq \sum_{y\geq y_\epsilon} a_n(y) &\preceq  
\underbrace{\sup_{n \geq 1} \sqrt{[ns_1]}  \Vert \mathcal H_{[ns_1]}\Vert_{\mathcal B_\delta}}_{<+\infty}\ 
\underbrace{ \sup_{n \geq 1} \vert {\bf h}'_{[nt_1]-[ns_1]} {\bf \mathbb{1}}_{[y_\epsilon, +\infty[ }\vert _{\mathcal B_\delta}}_{ \preceq {1\over y_\epsilon^\delta}}.
\end{align*}
We conclude choosing $y_\epsilon$ large enough.

 \underline{Condition $(iii)$.}  By (\ref{ekhrf}), it holds  with 
 $\displaystyle
 b 
 = \int_0^{+\infty}\varphi_1(u) u  \ \frac{\exp\left(\dfrac{-u^2}{2(t_1-s_1)}\right)}{t_1-s_1}du
 .$


\subsection{Convergence of the multidimensional distributions} 

We focus here  on the case $m=2$; the cases $m\geq 3$ is done by induction.
We fix $0 < t_1< t_2$ and, for $n\geq 1$ given,  let $\kappa=\kappa_{t_1}$ be the first crossing time after time $[nt_1]$ defined by
  $$\kappa:=\min\{k>[nt_1]: X_{[nt_1]} X_k\leq 0\}.$$
  As in the case $m=1$, it holds $$\mathbb{E}_x[\varphi_1(X^{(n)}(t_1))\varphi_2(X^{(n)}(t_2))] \approx \mathbb{E}_x\left[\varphi_1\left(\dfrac{X_{[nt_1]}}{\sigma\sqrt{n}}\right)\varphi_2\left(\dfrac{X_{[nt_2]}}{\sigma\sqrt{n}}\right)\right],$$ 
  and the right hand side term may be decomposed as $A_0(n)+A^{\pm}_1(n)+ A^{\pm}_2(n)$,
where
$$ A_0(n):= \mathbb{E}_x\left[\varphi_1\left(\dfrac{X_{[nt_1]}}{\sigma\sqrt{n}}\right)\varphi_2\left(\dfrac{X_{[nt_2]}}{\sigma\sqrt{n}}\right), X_{[nt_1]}=0\right],
$$
$$A^{\pm}_1(n):=  \sum_{k_2=[nt_1]+1}^{[nt_2]}\mathbb{E}_x\left[\varphi_1\left(\dfrac{X_{[nt_1]}}{\sigma\sqrt{n}}\right)\varphi_2\left(\dfrac{X_{[nt_2]}}{\sigma\sqrt{n}}\right) \mathbb{1}_{[\kappa=k_2]}\mathbb{1}_{[\pm X_{[n{t_1}]}>0]}\right],$$
and
$$A^{\pm}_2(n):=\mathbb{E}_x\left[\varphi_1\left(\dfrac{X_{[nt_1]}}{\sigma\sqrt{n}}\right)\varphi_2\left(\dfrac{X_{[nt_2]}}{\sigma\sqrt{n}}\right) \mathbb{1}_{[\kappa>[nt_2]]}\mathbb{1}_{[\pm X_{[n{t_1}]}>0]}\right].$$
As previously, the term $A_0(n)$ tends to $0$ since $(X_n)_{n \geq 0}$ is null recurrent.

\noindent $\bullet$ Estimate of $A^\pm_1(n)$
\begin{align*}
    A^+_1(n) \approx  \sum_{k_1=1}^{[nt_1]-1} &\displaystyle \sum_{{k_2=[nt_1]+1}}^{[nt_2]}  \sum_{\ell \geq 1}  \sum_{y \geq 1}  \sum_{z \geq 1} \sum_{w\leq 0} \mathbb{E}_x \bigg[\varphi_1\left(\dfrac{X_{[nt_1]}}{\sigma\sqrt{n}}\right)\varphi_2\left(\dfrac{X_{[nt_2]}}{\sigma\sqrt{n}}\right), C_\ell=k_1,\\
    & \hspace{1.5cm} X_{k_1}=y, y+\xi'_{k_1+1}>0, \ldots+, y+\xi'_{k_1+1}+\ldots+ \xi'_{k_2-2}>0,\\
    & \hspace{4.2cm} y+\xi'_{k_1+1}+\ldots + \xi'_{k_2-1}=z,  y+\xi'_{k_1+1}+\ldots + \xi'_{k_2}=w \bigg]\\
   =\displaystyle \sum_{k_1=1}^{[nt_1]-1} &\displaystyle \sum_{{k_2=[nt_1]+1}}^{[nt_2]}  \sum_{\ell \geq 1}  \sum_{y \geq 1}  \sum_{z\geq 1}  \sum_{w\leq 0} \mathbb{E}_x \bigg[\varphi_1\left(\dfrac{y+\xi'_{k_1+1}+\ldots+ \xi'_{[nt_1]}}{\sigma\sqrt{n}}\right)\varphi_2\left(\dfrac{X_{[nt_2]}}{\sigma\sqrt{n}}\right),\\
   & \hspace{8mm}C_\ell=k_1, X_{k_1}=y, y+\xi'_{k_1+1}>0, \ldots, y+\xi'_{k_1+1}+\ldots+ \xi'_{k_2-2}>0,\\ 
   & \hspace{5cm} y+\xi'_{k_1+1}+\ldots+ \xi'_{k_2-1}=z\bigg] \mathbb P[ \xi'_{k_2}=w-z]\\
     =\displaystyle \sum_{k_1=1}^{[nt_1]-1} &\displaystyle \sum_{y\geq 1} H_{k_1}(x,y)  \sum_{k_2=[nt_1]+1}^{[nt_2]}  \sum_{z\geq 1}
     \ \mathbb{E} \bigg[\varphi_1\left(\dfrac{y+S'_{[nt_1]-k_1}}{\sigma\sqrt{n}}\right) ,  \tau^{S'}(y)>k_2-k_1-1,\\
     &\hspace{4cm} y+S'_{k_2-k_1-1}=z\bigg]   \sum_{w\leq 0} \mathbb{E}_w\bigg[\varphi_2\left(\dfrac{X_{[nt_2]-k_2}}{\sigma \sqrt{n}}\right)\bigg] \mu'(w-z).
\end{align*}
For any $ (s_1, s_2) \in [0,  t_1] \times [t_1,  t_2]$ and $n \geq 1$, let $g_n$ be the function defined by
\begin{align*}
   g_n(s_1,s_2)&= n^2 \displaystyle \sum_{y\geq 1} H_{[ns_1]}(x,y)   \sum_{z\geq 1}\mathbb{E} \bigg[\varphi_1\left(\dfrac{y+S'_{[nt_1]-[ns_1]}}{\sigma\sqrt{n}}\right), \tau^{S'}(y)>[ns_2]-[ns_1]-1,    \\
     &\hspace{3.2cm} y+S'_{[ns_2]-[ns_1]-1}=z\bigg]   
     \sum_{w\leq 0} \mathbb{E}_w\bigg[\varphi_2\left(\dfrac{X_{[nt_2]-[ns_2]}}{\sigma \sqrt{n}}\right)\bigg] 
     \mu'(w-z) 
\end{align*}
if $0\leq s_1 < \frac{[nt_1]}{n}$  and  ${{[nt_1]+1}\over n}\leq s_2\leq {[nt_2]\over n}$, and $0$ otherwise.
 Hence, 
 $$A^+_1(n)=  \int_{0}^{t_1} \int_{t_1}^{t_2} g_n(s_1, s_2) ds_1 ds_2 + O\left(\dfrac{1}{\sqrt{n}}\right).$$
 
  {The convergence of the term $A_1^+(n)$  as $n\to +\infty$ is a consequence of the two following properties whose proofs are postponed at the end of the present subsection:

\noindent 
$\bullet$ for any $n \geq 1$,
\begin{equation} \label{iehutvngjzo}
  |g_n(s_1, s_2)| \preceq\frac{1+|x|}{\sqrt{s_1(s_2-s_1)^3}} \in L^1([0, t_1] \times [t_1, t_2]);
\end{equation}
$\bullet$ 
for any $(s_1, s_2) \in [0, t_1]\times [t_1, t_2]$,   
\begin{equation} \label{vahbj}
 \lim_{n \to +\infty}g_n(s_1, s_2) =\frac{\gamma}{\pi^2}\displaystyle \int_{0}^{+\infty}  \int_{-\infty}^{+\infty}\varphi_1(u)\Tilde{\varphi_2}(v) u^2  \frac{e^{\frac{-v^2}{2(t_2-s_2)}} e^{\frac{-u^2}{2\frac{(t_1-s_1)(s_2-t_1)}{s_2-s_1}}}}{\sqrt{s_1 (t_1-s_1)^3(s_2-t_1)^3 (t_2-s_2)}}du dv.\end{equation} 
Indeed, applying the Lebesgue dominated convergence theorem, we obtain 
\begin{align*}
   &\lim_{n\to +\infty}A^+_1(n)\\
  &=\dfrac{\gamma}{\pi^2}\displaystyle \int_{0}^{+\infty}  \int_{-\infty}^{+\infty}\varphi_1(u)\Tilde{\varphi_2}(v)\bigg(\displaystyle \int_0^{t_1} \int_{t_1}^{t_2}\frac{e^{\frac{-v^2}{2(t_2-s_2)}} u^2 \exp\bigg(\frac{-u^2}{2\frac{(t_1-s_1)(s_2-t_1)}{s_2-s_1}}\bigg)}{\sqrt{s_1 (t_1-s_1)^3(s_2-t_1)^3 (t_2-s_2)}}ds_1 ds_2\bigg)du dv\\
  &=\dfrac{\gamma}{\pi^2}\frac{\sqrt{2\pi t_1}}{t_1}\displaystyle \int_{0}^{+\infty}  \int_{-\infty}^{+\infty}\varphi_1(u)\Tilde{\varphi_2}(v)|u| \bigg(\displaystyle \int_{t_1}^{t_2}\frac{e^{\frac{-u^2 s_2}{2t_1(s_2-t_1)}}e^{\frac{-v^2}{2(t_2-s_2)}}}{\sqrt{(t_2-s_2)(s_2-t_1)^3}}ds_2 \bigg)du dv\\
  &=\dfrac{2\gamma}{\pi\sqrt{t_1(t_2-t_1)}}\displaystyle \int_{0}^{+\infty}  \int_{-\infty}^{+\infty}\varphi_1(u)\Tilde{\varphi_2}(v) e^{-\frac{u^2t_2+v^2t_1+2|uv|t_1}{t_1(t_2-t_1)}}du dv
\end{align*}
which  can be rewritten as
\begin{align} \label{zkehvrnpejv}
   \lim_{n\to +\infty}A^+_1(n)&=\dfrac{2\gamma^2}{\pi\sqrt{t_1(t_2-t_1)}}\displaystyle \int_{0}^{+\infty}  \int_0^{+\infty}\varphi_1(u)\varphi_2(v)e^{-\frac{u^2}{2t_1}} e^{-\frac{(u+v)^2}{2(t_2-t_1)}}du dv \notag\\
  &+\dfrac{2\gamma(1-\gamma)}{\pi\sqrt{t_1(t_2-t_1)}}\displaystyle \int_{0}^{+\infty}  \int_{-\infty}^0\varphi_1(u) \varphi_2(v) e^{-\frac{u^2}{2t_1}} e^{-\frac{(u-v)^2}{2(t_2-t_1)}}du dv, 
\end{align} 
by using the classical integral
$\displaystyle 
 \int_0^{+\infty}\dfrac{1}{\sqrt{x}}\exp\left(-\lambda_1 x -\dfrac{\lambda_2}{x}\right)dx= \sqrt{\dfrac{\pi}{\lambda_1}}e^{-2\sqrt{\lambda_1 \lambda_2}}
 $   for any $\lambda_1>0$ and $\lambda_2\geq 0$. 
 
\noindent The same argument holds for the term $A_1^-(n)$ and yields
 \begin{align} \label{zjbeVHKZ}
   \lim_{n\to +\infty}A^-_1(n)&=\dfrac{2(1-\gamma)^2}{\pi\sqrt{t_1(t_2-t_1)}}\displaystyle \int_{-\infty}^0  \int_{-\infty}^0\varphi_1(u)\varphi_2(v)e^{-\frac{u^2}{2t_1}} e^{-\frac{(u+v)^2}{2(t_2-t_1)}}du dv
   \notag\\
  &+\dfrac{2\gamma(1-\gamma)}{\pi\sqrt{t_1(t_2-t_1)}}\displaystyle \int_{-\infty}^0  \int_0^{+\infty}\varphi_1(u) \varphi_2(v) e^{-\frac{u^2}{2t_1}} e^{-\frac{(u-v)^2}{2(t_2-t_1)}}du dv.
\end{align}
}

\noindent $\bullet$ Estimate of $A^+_2(n)$ 
\begin{align*}
    A^+_2(n)= \sum_{k=1}^{[nt_1]-1} & \sum_{\ell \geq 1}  \sum_{y \geq 1} \mathbb{E}_x \bigg[\varphi_1\left(\dfrac{X_{[nt_1]}}{\sigma\sqrt{n}}\right)\varphi_2\left(\dfrac{X_{[nt_2]}}{\sigma\sqrt{n}}\right), C_\ell=k, X_k=y,\\ 
    &y+\xi'_{k+1}>0, \ldots, y+\xi'_{k+1}+\ldots + \xi'_{[nt_1]}>0,\ldots, y+\xi'_{k+1}+\ldots + \xi'_{[nt_2]}>0 \bigg]\\  
    = \sum_{k=1}^{[nt_1]-1} & \sum_{y \geq 1}  H_k(x,y) \mathbb{E}  \bigg[\varphi_1\left(\dfrac{y+S'_{[nt_1]-k}}{\sigma\sqrt{n}}\right)\varphi_2\left(\dfrac{y+S'_{[nt_2]-k}}{\sigma\sqrt{n}}\right), \tau^{S'}(y)>[nt_2]-k\bigg].  
\end{align*}
For any $n \geq 1$, let $g_n: r \mapsto g_n(r)$ be the real function defined on $[0, t_1]$ by
\begin{align*}
g_n(r):= n \sum_{y \geq 1} H_{[nr]}(x,y) \mathbb{E}  \bigg[\varphi_1\left(\dfrac{y+S'_{[nt_1]-[nr]}}{\sigma\sqrt{n}}\right)&\varphi_2\left(\dfrac{y+S'_{[nt_2]-[nr]}}{\sigma\sqrt{n}}\right), \tau^{S'}(y)>[nt_2]-[nr]\bigg]\\
\end{align*}
if $0\leq r < \frac{[nt_1]}{n}$  and $g_n(r)=0$ if $\frac{[nt_1]}{n}\leq r\leq t_1 $.\\
In the same way as  above, we set
$$a_n(y):= n  H_{[nr]}(x,y) \mathbb{P}[\tau^{S'}(y)>[nt_2]-[nr]]$$
and
$$b_n(y):= \mathbb{E}  \bigg[\varphi_1\left(\dfrac{y+S'_{[nt_1]-[nr]}}{\sigma\sqrt{n}}\right)\varphi_2\left(\dfrac{y+S'_{[nt_2]-[nr]}}{\sigma\sqrt{n}}\right) \mid \tau^{S'}(y)>[nt_2]-[nr]\bigg].$$
Sequences $(a_n(y))_{y \geq 1}$ and $(b_n(y))_{y \geq 1}$ satisfy assumptions of Lemma \ref{product}. Indeed, the limit of $
\displaystyle \sum_{y \geq 1} a_n(y)$ is given by \eqref{sum a_n(y)} and condition $(ii)$ of this lemma has been checked  previously. Furthermore, by   Theorem $3.2$ in \cite{Bolthausen} and Theorems $2.23$ and $3.4$ in \cite{Iglehart}, it holds 
\begin{align*}
     \lim_{n\to +\infty} b_n(y)&=  \lim_{n\to +\infty} \mathbb{E}  \bigg[\varphi_1\left(\dfrac{y+S'_{[nt_1]-[nr]}}{\sigma\sqrt{[nt_2]-[nr]}}\dfrac{\sqrt{[nt_2]-[nr]}}{\sqrt{n}}\right)\\
    &\hspace{2cm} \varphi_2\left(\dfrac{y+S'_{[nt_2]-[nr]}}{\sigma\sqrt{n}}\frac{\sqrt{[nt_2]-[nr]}}{\sqrt{n}}\right) \mid \tau^{S'}(y)>[nt_2]-[nr]\bigg]\\
    &=\dfrac{1}{\sqrt{2\pi (t_2-t_1)}} \int_0^{+\infty} \int_0^{+\infty} \varphi_1(u) \varphi_2(v)\frac{\sqrt{t_2-r}}{\sqrt{(t_1-r)^3}} ue^{\frac{-u^2}{2(t_1-r)}}\\
    &\hspace{5.7cm}\times\bigg(e^{-\frac{(u-v)^2}{2(t_2-t_1)}}-e^{\frac{-(u+v)^2}{2(t_2-t_1)}}\bigg) du dv.
\end{align*}
It immediately yields 
 \begin{align} \label{JKQBHCFIZGJ}
      \lim_{n\to +\infty} A^+_2(n)&=\frac{\gamma}{\pi \sqrt{2\pi (t_2-t_1)}}  \int_0^{t_1}\bigg(\dfrac{1}{\sqrt{2\pi (t_2-t_1)}} \int_0^{+\infty} \int_0^{+\infty} \varphi_1(u) \varphi_2(v)\notag\\     &\hspace{2cm}\times\frac{\sqrt{t_2-r}}{\sqrt{r(t_2-r)(t_1-r)^3}}ue^{\frac{-u^2}{2(t_1-r)}} \bigg(e^{-\frac{(u-v)^2}{2(t_2-t_1)}}-e^{\frac{-(u+v)^2}{2(t_2-t_1)}}\bigg) du dv \bigg) dr\notag\\
     &=\dfrac{\gamma}{\pi \sqrt{t_1(t_2-t_1)}}  \int_0^{+\infty} \int_0^{+\infty} \varphi_1(u) \varphi_2(v)e^{\frac{-u^2}{2t_1}}\bigg(e^{-\frac{(u-v)^2}{2(t_2-t_1)}}-e^{\frac{-(u+v)^2}{2(t_2-t_1)}}\bigg) du  dv.
 \end{align}
 
\noindent Analogously, one gets
\begin{equation} \label{qzjkbehvgn}
\lim_{n\to +\infty} A^-_2(n)=\dfrac{1-\gamma}{\pi \sqrt{t_1(t_2-t_1)}}  \int_{-\infty}^0 \int_{-\infty}^0 \varphi_1(u) \varphi_2(v)e^{\frac{-u^2}{2t_1}}\bigg(e^{-\frac{(u-v)^2}{2(t_2-t_1)}}-e^{-\frac{(u+v)^2}{2(t_2-t_1)}}\bigg) du  dv
\end{equation}
Combining  (\ref{zkehvrnpejv}), 
(\ref{zjbeVHKZ}), 
(\ref{JKQBHCFIZGJ})
and
(\ref{qzjkbehvgn}), we conclude 
$$ \lim_{n\to +\infty} \mathbb{E}_x \left[\varphi_1(X^{(n)}(t_1))\varphi_2(X^{(n)}(t_2))\right]=  \int_{-\infty}^{+\infty} \int_{-\infty}^{+\infty}\varphi_1(u)\varphi_2(v) p^{\gamma}_{t_1}(0, u)p^{\gamma}_{t_2-t_1}(u,v) du dv. $$

\noindent {\bf Proof of properties (\ref{iehutvngjzo}) and  (\ref{vahbj})}

Following the same strategy as in the one dimensional case,  we decompose $g_n(s_1, s_2)$ as $g_n(s_1, s_2)= \displaystyle \sum_{y \geq 1} \sum_{z \geq 1}  \sum_{w \le 0} a_n(y, z, w)b_n(y, z,w)$, where  
\begin{align*}
a_n(y, z, w):=   n^2  H_{[ns_1]}(x,y) \mathbb{P}[\tau^{S'}(y)>[ns_2]-[ns_1]-1, y+S'_{[ns_2]-[ns_1]-1}=z]\ \mu'(w-z)
\end{align*}
and
\begin{align*}
    b_n(y, z, w):= &\mathbb{E} \bigg[\varphi_1\left(\dfrac{y+S'_{[nt_1]-[ns_1]}}{\sigma\sqrt{n}}\right) \mid \tau^{S'}(y)>[ns_2]-[ns_1]-1, y+S'_{[ns_2]-[ns_1]-1}=z\bigg]
   \\
    & \qquad \times   \mathbb{E}_w\bigg[\varphi_2\left(\dfrac{X_{[nt_2]-[ns_2]}}{\sigma \sqrt{n}}\right)\bigg].
\end{align*}
Properties  (\ref{iehutvngjzo}) and  (\ref{vahbj}) are a direct consequence of Lemma \ref{product}, applied to  the families $(a_n(y, z,w))_{y, z\geq 1, w\leq 0}, (b_n(y, z, w))_{y, z\geq 1, w\leq 0 }$; it thus sufficies  to  check that conditions $(i), (ii)$ and $(iii)$ of this lemma are satisfied in this new situation.

 \underline{ Condition $(i)$.} The sum $\displaystyle  \sum_{y \geq 1} \sum_{z \geq 1}  \sum_{w \le 0} a_n(y, z, w)  $ may be written as
\begin{equation}\label{czjvhegjfbhcj}
\sum_{y \geq 1} \sum_{z \geq 1}  \sum_{w \le 0} a_n(y, z, w) = {1+o(n)\over \sqrt{s_1 (s_2- s_1)^3}}   \sqrt{[ns_1]}\mathcal H_{[ns_1]}({\bf b}'_{[ns_2]-[ns_1]})(x),
\end{equation}
where we   set
$
{\bf b}'_{k}(y)= k^{3/2} \mathbb P[\tau^{S'}(y) =k].
$

As previously,   the sequence  $(\sqrt{[ns_1]}\mathcal H_{[ns_1]})_{n\geq 1}$ converges in $(\mathcal L({{\mathcal B_\delta}}), \Vert \cdot \Vert_{\mathcal B_\delta})$ to the operator ${1\over 2\pi\bigl(c\nu( \check{h}_a)+c'\nu(h'_d)\bigr)}\Pi$; furthermore, the sequence $({\bf b}'_{[ns_2]-[ns_1]})_{n \geq 1}$ converges  in ${\mathcal B_\delta}$ to the function $c'h'_d$.
Hence condition $(i)$ holds  with 
$$A= \dfrac{1}{\sqrt{s_1(s_2-s_1)^3}}\dfrac{c'\nu(h'_d)}{2\pi\bigl(c\nu( \check{h}_a)+c'\nu(h'_d)\bigr)}= \dfrac{\gamma}{2\pi \sqrt{s_1(s_2-s_1)^3}}.$$


 \underline{Condition $(ii)$.}  Fix $\epsilon >0$ and $y_\epsilon\geq 1$.  As above ${\bf h}'_k {\bf \mathbb{1}}_{[y_\epsilon, +\infty[ }$,  the  function  ${\bf b}'_k{\bf \mathbb{1}}_{[y_\epsilon, +\infty)}$ belongs to ${\mathcal B_\delta}$ and $  \vert  {\bf b}'_k{\bf \mathbb{1}}_{[y_\epsilon, +\infty)}\vert _{\mathcal B_\delta} \leq {C_1\over y_\epsilon^\delta}$  for some constant $C_1>0$.
By  (\ref{czjvhegjfbhcj}), it follows
    \begin{align*}
 0\leq \sum_{y\geq y_\epsilon, z \geq 1, w\leq 0} a_n(y, z, w) &\preceq     \sqrt{[ns_1]}\mathcal H_{[ns_1]}({\bf b}'_{[ns_2]-[ns_1]}{\bf \mathbb{1}}_{[y_\epsilon, +\infty)})(x),
 \\
 & \preceq  
\underbrace{\sup_{n \geq 1}  \Vert \sqrt{[ns_1]}\mathcal H_{[ns_1]}\Vert_{\mathcal B_\delta}}_{<+\infty}\  
\underbrace{ \sup_{n \geq 1} \vert {\bf b}'_{[ns_2]-[ns_1]}{\bf \mathbb{1}}_{[y_\epsilon, +\infty)}\vert _{\mathcal B_\delta}}_{ \preceq {1\over y_\epsilon^\delta}}.
\end{align*}
This last  right hand side term  is $<\epsilon$ for sufficiently large $y_\epsilon$.

 Furthermore,   $0\leq a_n(y, z, w) \preceq (1+y)(1+z) \mu'(w-z)$ for any  fixed $y, z,  w$ and any $n \geq 1$; hence, for any $0\leq y < y_\epsilon,$ it holds  $\displaystyle \sum_{z +\vert w\vert >t}  a_n(y, z, w) <\epsilon$ if $t$ is large enough since $\sum_{z\geq 1}z\mu'(]-\infty , -z]) <+\infty$. This completes the argument.

{\underline{Condition $(iii)$.}  On the one hand, by (\ref{two}),  for any $y, z \geq 1$,
\begin{align*}
    & \lim_{n\to +\infty}\mathbb{E} \bigg[\varphi_1\left(\dfrac{y+S'_{[nt_1]-[ns_1]}}{\sigma\sqrt{n}}\right) \mid \tau^{S'}(y)>[ns_2]-[ns_1]-1, y+S'_{[ns_2]-[ns_1]-1}=z\bigg]\\
    &\hspace{2cm}=  \int_0^{+\infty} 2\varphi_1(u'\sqrt{s_2-s_1})  \exp\left(\dfrac{-{u'}^2}{2\frac{t_1-s_1}{s_2-s_1} \frac{s_2-t_1}{s_2-s_1}}\right) \frac{{u'}^2}{\sqrt{2\pi \frac{(t_1-s_1)^3}{(s_2-s_1)^3} \frac{(s_2-t_1)^3}{(s_2-s_1)^3}}}du'\\
    &\hspace{2cm}=\dfrac{2}{\sqrt{2\pi}} \int_0^{+\infty} \varphi_1(u)  \exp\left(\dfrac{-u^2}{2\frac{(t_1-s_1)(s_2-t_1)}{s_2-s_1}}\right) \frac{u^2}{\sqrt{\frac{(t_1-s_1)^3(s_2-t_1)^3}{(s_2-s_1)^3}}}du.
\end{align*}
}
On the other hand, the  one dimensional case  $m=1$  studied above yields, for any $w \leq 0$,
$$  \lim_{n\to +\infty}   \mathbb{E}_w\bigg[\varphi_2\left(\dfrac{X_{[nt_2]-[ns_2]}}{\sigma \sqrt{n}}\right)\bigg]  = \int_{\mathbb{R}}\Tilde{\varphi_2}(v)\frac{2\exp\bigg(\dfrac{-v^2}{2(t_2-s_2)}\bigg)}{\sqrt{2\pi (t_2-s_2)}}dv$$
with $\Tilde{\varphi}_2(v)= \gamma \varphi_2(v)\mathbb{1}_{(0,+\infty)}(v)+(1-\gamma) \varphi_2(v)\mathbb{1}_{(-\infty,0)}(v).$

%
%
%
\vspace{3mm}
  
\subsection{Tightness of the sequence \texorpdfstring{$\{ X^{(n)}\}_{n \geq 1}$}{X}} 
Let us recall the modulus of continuity of a function $f:[0,T] \to \mathbb{R}$ is defined by
$$\omega_f(\delta):= \sup\{|f(t)-f(s)|:t, s\in [0,T] \,s.t.\, |t-s|\leq \delta  \}.$$
 By Theorems $7.1$ and $7.3$ in \cite{Billingsley}, we have to show that the following conditions hold:
\begin{enumerate}[(i)]
    \item For every $\eta>0$, there exist $a>0$ and $n_\eta\geq 1$ such that
    $$\mathbb P[|X^{(n)}(0)|\geq a] \leq \eta, \quad \forall n\geq n_\eta.$$
    
    \item  For every $\epsilon>0$ and $\eta>0$, there exist $\delta\in (0,1)$ and $n_{\epsilon, \eta}\geq 1$ such that
    $$\mathbb P[\omega_{X^{(n)}}(\delta)\geq \epsilon] \leq \eta, \quad \forall n\geq n_{\epsilon, \eta}.$$
\end{enumerate}
\begin{proof}
 Condition (i) is obviously satisfied.\\
 Let us now check the condition (ii). Set $I_{n,\delta}:= \{(i, j) \in \mathbb N \mid 1\leq i <j \leq n \text{ and } |i-j| \leq n\delta\}$ and note that we have
    \begin{align}\label{modulus bound}
        \omega_{X^{(n)}}(\delta) \leq \dfrac{7}{\min\{\sigma, \sigma'\}\sqrt{n}}\left(\displaystyle \sup_{(i,j)\in I_{n,\delta}}|S_i-S_j|+\sup_{(i,j)\in I_{n,\delta}}|S'_i-S'_j|\right).
    \end{align}
We suggest the following figure as a useful illustration of this bound. 
\begin{figure}[H]
    \centering
    \begin{subfigure}[b]{1.02\textwidth} \includegraphics[width=\textwidth]{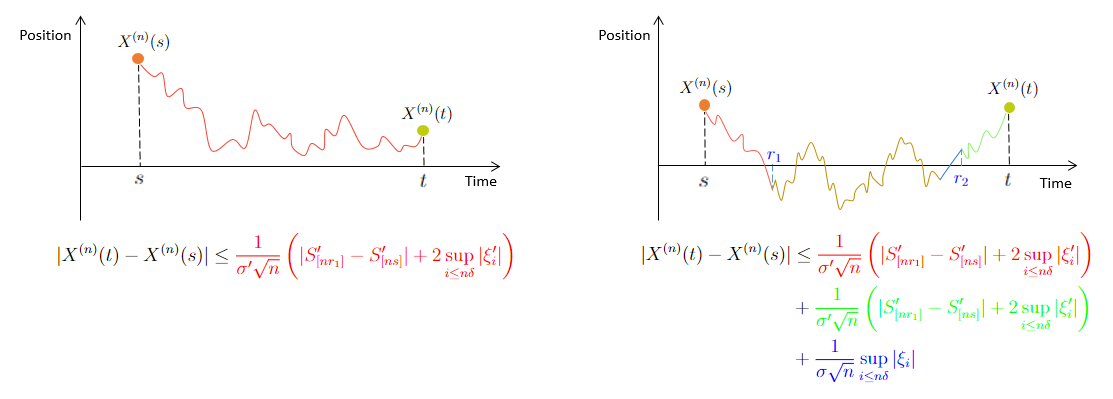}
    \end{subfigure}
    \begin{subfigure}[b]{1.02\textwidth}
\includegraphics[width=\textwidth]{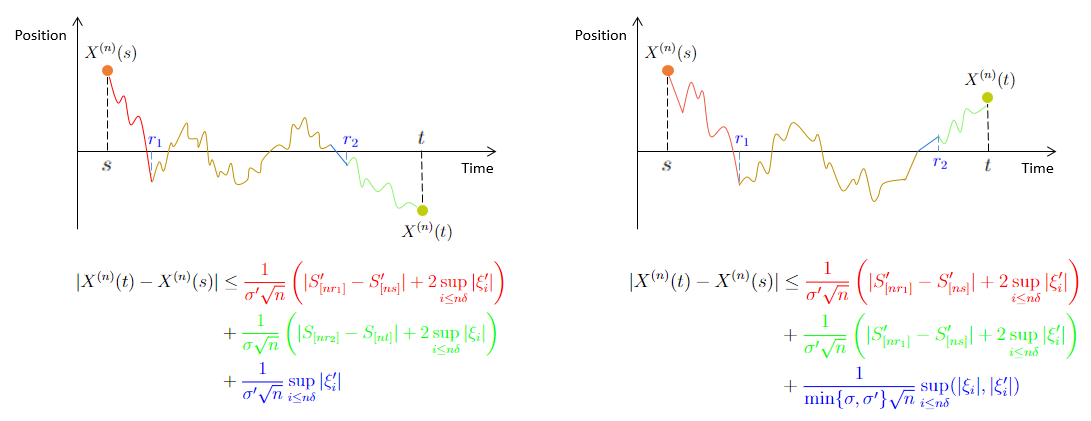}
    \end{subfigure}
    \caption{Fluctuation of the normalized oscillating random walk on the time interval $[s, t]$ with its first and last crossing times $r_1$ and $r_2$, respectively.}
\end{figure}
\noindent Moreover, by \cite{Billingsley} (see Chapter $7$) one also gets
    \begin{align}\label{tight}
        \lim_{\delta \to 0} \lim_{n \to +\infty}\mathbb P\left[\dfrac{1}{\sigma\sqrt{n}} \sup_{(i,j)\in I_{n,\delta}}|S_i-S_j| \geq \epsilon\right]= \lim_{\delta \to 0} \lim_{n \to +\infty}\mathbb P\left[\dfrac{1}{\sigma'\sqrt{n}} \sup_{(i,j)\in I_{n,\delta}}|S'_i-S'_j| \geq \epsilon\right]=0.
    \end{align}
    The condition (ii) immediately follows by \eqref{modulus bound} and \eqref{tight}. Hence, we conclude that the sequence $\{X^{(n)}\}_{n\geq 1}$ is tight.
 \end{proof}


\end{document}